\documentclass[reqno,fleqn]{amsart}
\usepackage{amssymb}
\usepackage{amsmath}
\usepackage{amsfonts}
\usepackage{fancyhdr}

\newtheorem{theorem}{Theorem}
\theoremstyle{plain}

\newtheorem{corollary}{Corollary}

\newtheorem{definition}{Definition}
\newtheorem{example}{Example}

\newtheorem{problem}{Problem}
\newtheorem{proposition}{Proposition}
\newtheorem{remark}{Remark}

\input{tcilatex}

\begin{document}
\title[On the $L^{p}-$theory of Anisotropic singular...]{On the $L^{p}-$%
theory of Anisotropic singular perturbations of elliptic problems}
\author{Chokri Ogabi}
\dedicatory{ \ \ Academie de Grenoble, 38300 France}
\email{chokri.ogabi@ac-grenoble.fr}
\date{28,January 2015}
\keywords{Anisotropic singular perturbations, elliptic problem, $L^{p}-$%
theory, entropy solutions. Asymptotic behaviour. rate of convergence}
\subjclass{35J15, 35B60, 35B25}
\maketitle

\begin{abstract}
In this article we give an extention of the $L^{2}-$theory of anisotropic
singular perturbations for elliptic problems. We study a linear and some
nonlinear problems involving $L^{p}$ data ($1<p<2$). Convergences in pseudo
Sobolev spaces are proved for weak and entropy solutions, and rate of
convergence is given in cylindrical domains
\end{abstract}

\section{\textbf{Introduction}}

\subsection{Preliminaries}

In this article we shall give an extension of the $L^{2}-$theory of the
asymptotic behavior of elliptic, anisotropic singular perturbations
problems. This kind of singular perturbations has been introduced by M.
Chipot \cite{12}. From the physical point of view, these problems can
modelize diffusion phenomena when the diffusion coefficients in certain
directions are going toward zero. The $L^{2}$ theory of the asymptotic
behavior of these problems has been studied by M. Chipot and many
co-authors. First of all, let us begin by a brief discussion on the
uniqueness of the weak solution ( by weak a solution we mean a solution in
the sense of distributions) to the problem 
\begin{equation}
\left\{ 
\begin{array}{cc}
-\func{div}(A\nabla u)=f\text{ \ \ \ \ \ \ \ \ \ \ } &  \\ 
u=0\text{ \ \ on }\partial \Omega \text{\ \ \ \ \ \ \ \ \ \ \ \ \ \ \ \ } & 
\text{ }%
\end{array}%
\right.  \label{1}
\end{equation}

where $\Omega \subset 
%TCIMACRO{\U{211d} }%
%BeginExpansion
\mathbb{R}
%EndExpansion
^{N}$, $N\geq 2$ is\ a bounded Lipschitz domain, we suppose that $f\in
L^{p}(\Omega )$ ($1<p<2$). The diffusion matrix $A=(a_{ij})$ is supposed to
be bounded and satisfies the ellipticity assumption on $\Omega $ ( see
assumptions (\ref{2}) and (\ref{3}) in subsection 1.2). It is well known
that (\ref{1}) has at least a weak solution in $W_{0}^{1,p}(\Omega )$.
Moreover, if $A$ is symmetric and continuous and $\partial \Omega \in C^{2}$ 
\cite{1} then (\ref{1}) has a unique solution in $W_{0}^{1,p}(\Omega )$. If $%
A$ is discontinuous the uniqueness assertion is false, in \cite{4} Serrin
has given a counterexample when $N\geq 3$. However, if $N=2$ and if $%
\partial \Omega $ is sufficiently smooth and without any continuity
assumption on $A$, (\ref{1}) has a unique weak solution in $%
W_{0}^{1,p}(\Omega )$. The proof is based on the Meyers regularity theorem
(see for instance \cite{8}). To treat this pathology, Benilin, Boccardo,
Gallouet, and al have introduced the concept of the entropy solution \cite{2}
for problems involving $L^{1}$ data (or more generally a Radon measure).

For every $k>0$ We define the function $T_{k}:%
%TCIMACRO{\U{211d} }%
%BeginExpansion
\mathbb{R}
%EndExpansion
\rightarrow 
%TCIMACRO{\U{211d} }%
%BeginExpansion
\mathbb{R}
%EndExpansion
$ by 
\begin{equation*}
T_{k}(s)=\left\{ 
\begin{array}{cc}
s & ,\left\vert s\right\vert \leq k \\ 
ksgn(x) & \left\vert s\right\vert \geq k%
\end{array}%
\right.
\end{equation*}

And we define the space $\mathcal{T}_{0}^{1,2}$ introduced in \cite{2}.%
\begin{equation*}
\mathcal{T}_{0}^{1,2}(\Omega )=\left\{ 
\begin{array}{c}
u\text{ }:\Omega \rightarrow 
%TCIMACRO{\U{211d} }%
%BeginExpansion
\mathbb{R}
%EndExpansion
\text{ measurable such that for any }k>0\text{ there exists} \\ 
\text{ }(\phi _{n})\subset H_{0}^{1}(\Omega ):\phi _{n}\rightarrow T_{k}(u)%
\text{ a.e in }\Omega \text{ } \\ 
\text{and }(\nabla \phi _{n})_{n\in 
%TCIMACRO{\U{2115} }%
%BeginExpansion
\mathbb{N}
%EndExpansion
}\text{ is bounded in }L^{2}(\Omega )%
\end{array}%
\right\}
\end{equation*}

This definition of $\mathcal{T}_{0}^{1,2}$ is equivalent to the original one
given in \cite{2}.In fact, this is a characterization of this space \cite{2}%
. Now, more generally, for $f\in L^{1}(\Omega )$ we have the following
definition of entropy solution \cite{2}.

\begin{definition}
\bigskip A function $u\in \mathcal{T}_{0}^{1,2}(\Omega )$ is said to be an
entropy solution to (\ref{1}) if \ \ 
\begin{equation*}
\int_{\Omega }A\nabla u\cdot \nabla T_{k}(u-\varphi )dx\leq \int_{\Omega
}fT_{k}(u-\varphi )dx\text{, }\varphi \in \mathcal{D(}\Omega \mathcal{)}%
\text{, }k>0
\end{equation*}
\end{definition}

We refer the reader to \cite{2} for more details about the sense of this
formulation. The main results of \cite{2} show that (\ref{1}) has a unique
entropy solution which is also a weak solution of (\ref{1}) moreover since $%
\Omega $ is bounded then this solution belongs to $\dbigcap\limits_{1\leq r<%
\frac{N}{N-1}}W_{0}^{1,r}(\Omega )$.

\subsection{Description of the problem and functional setting}

Throughout this article we will suppose that $f\in L^{p}(\Omega )$, $1<p<2$,
(we can suppose that $f\notin L^{2}(\Omega )$). We give a description of the
linear problem (some nonlinear problems will be studied later). Consider the
following singular perturbations problem%
\begin{equation}
\left\{ 
\begin{array}{cc}
-\func{div}(A_{\epsilon }\nabla u_{\epsilon })=f\text{ \ \ \ \ \ \ \ \ \ \ }
&  \\ 
u_{\epsilon }=0\text{ \ \ on }\partial \Omega \text{\ \ \ \ \ \ \ \ \ \ \ \
\ \ \ \ } & \text{ }%
\end{array}%
\right. ,  \label{2}
\end{equation}%
where $\Omega $ is a bounded Lipschitz domain of $%
%TCIMACRO{\U{211d} }%
%BeginExpansion
\mathbb{R}
%EndExpansion
^{N}$. Let $q\in 
%TCIMACRO{\U{2115} }%
%BeginExpansion
\mathbb{N}
%EndExpansion
^{\ast }$, $N-q\geq 2$. We denote by $x=(x_{1},...,x_{N})=(X_{1},X_{2})\in 
%TCIMACRO{\U{211d} }%
%BeginExpansion
\mathbb{R}
%EndExpansion
^{q}\times 
%TCIMACRO{\U{211d} }%
%BeginExpansion
\mathbb{R}
%EndExpansion
^{N-q}$ i.e. we split the coordinates into two parts. With this notation we
set%
\begin{equation*}
\nabla =(\partial _{x_{1}},...,\partial _{x_{N}})^{T}=\binom{\nabla _{X_{1}}%
}{\nabla _{X_{2}}},\text{ }
\end{equation*}

where 
\begin{equation*}
\nabla _{X_{1}}=(\partial _{x_{1}},...,\partial _{x_{q}})^{T}\text{ and }%
\nabla _{X_{2}}=(\partial _{x_{q+1}},...,\partial _{x_{N}})^{T}
\end{equation*}

Let $A=(a_{ij}(x))$ be a $N\times N$ matrix which satisfies the ellipticity
assumption%
\begin{equation}
\exists \lambda >0:A\xi \cdot \xi \geq \lambda \left\vert \xi \right\vert
^{2}\text{ }\forall \xi \in 
%TCIMACRO{\U{211d} }%
%BeginExpansion
\mathbb{R}
%EndExpansion
^{N}\text{ for a.e }x\in \Omega ,  \label{3}
\end{equation}

and 
\begin{equation}
a_{ij}(x)\in L^{\infty }(\Omega ),\forall i,j=1,2,....,N,\text{ }  \label{4}
\end{equation}

We have decomposed $A$ into four blocks%
\begin{equation*}
A=\left( 
\begin{array}{cc}
A_{11} & A_{12} \\ 
A_{21} & A_{22}%
\end{array}%
\right) ,
\end{equation*}

where $A_{11}$, $A_{22}$ are respectively $q\times q$ and $(N-q)\times (N-q)$
matrices. For $0<\epsilon \leq 1$ we have set 
\begin{equation*}
A_{\epsilon }=\left( 
\begin{array}{cc}
\epsilon ^{2}A_{11} & \epsilon A_{12} \\ 
\epsilon A_{21} & A_{22}%
\end{array}%
\right)
\end{equation*}

We denote $\Omega _{X_{1}}=\left\{ X_{2}\in 
%TCIMACRO{\U{211d} }%
%BeginExpansion
\mathbb{R}
%EndExpansion
^{N-q}:(X_{1},X_{2})\in \Omega \right\} $ and $\Omega ^{1}=P_{1}\Omega $
where $P_{1}:%
%TCIMACRO{\U{211d} }%
%BeginExpansion
\mathbb{R}
%EndExpansion
^{N}\rightarrow 
%TCIMACRO{\U{211d} }%
%BeginExpansion
\mathbb{R}
%EndExpansion
^{p}$ is the usual projector. We introduce the space 
\begin{equation*}
V_{p}=\left\{ 
\begin{array}{c}
u\in L^{p}(\Omega )\mid \nabla _{X_{2}}u\in L^{p}(\Omega )\text{,\ \ \ \ \ \
\ \ \ \ \ \ \ \ \ \ \ \ \ \ \ \ \ \ } \\ 
\text{and for a.e }X_{1}\in \Omega ^{1},u(X_{1},\cdot )\in
W_{0}^{1,p}(\Omega _{X_{1}})%
\end{array}%
\right\}
\end{equation*}

We equip $V_{p}$ with the norm 
\begin{equation*}
\left\Vert u\right\Vert _{V_{p}}=\left( \left\Vert u\right\Vert
_{L^{p}(\Omega )}^{p}+\left\Vert \nabla _{X_{2}}u\right\Vert _{L^{p}(\Omega
)}^{p}\right) ^{\frac{1}{p}},
\end{equation*}

then one can show easily that $(V_{p},\left\Vert \cdot \right\Vert _{V_{p}})$
is a separable reflexive Banach space.

The passage to the limit (formally) in (\ref{2}) gives the limit problem%
\begin{equation}
\left\{ 
\begin{array}{cc}
-\func{div}_{X_{2}}(A_{22}\nabla _{X_{2}}u_{0}(X_{1},\cdot ))=f\text{ }%
(X_{1},\cdot )\text{\ \ \ \ \ \ \ \ \ \ } &  \\ 
u_{0}(X_{1},\cdot )=0\text{ \ \ on }\partial \Omega _{X_{1}}\text{\ \ \ \ \
\ \ }X_{1}\in \Omega ^{1}\text{\ \ \ \ \ \ \ \ \ \ } & \text{ }%
\end{array}%
\right.  \label{5}
\end{equation}

The $L^{2}$-theory (when $f\in L^{2}$) of problem (\ref{2}) has been treated
in \cite{3}, convergence has been proved in $V_{2}$ and rate of convergence
in the $L^{2}-$norm has been given. For the $L^{2}-$theory of several
nonlinear problems we refer the reader to \cite{9},\cite{11},\cite{10}. This
article is mainly devoted to study the $L^{p}-$theory of the asymptotic
behavior of linear and nonlinear singularly perturbed problems. In other
words, we shall study the convergence $u_{\epsilon }\rightarrow u_{0}$ in$%
V_{p}$ (Notice that in \cite{9}, authors have treated some problems
involving $L^{p}$ data where some others data of the equations depend on $p$%
, one can check easily that it is not the $L^{p}$ theory which we expose in
this manuscript). Let us briefly summarize the content of the paper:

\begin{itemize}
\item In section 2: We study the linear problem, we prove convergences for
weak and entropy solutions.

\item In section 3: We give the rate of convergence in a cylindrical domain
when the data is independent of $X_{1}$.

\item In section 4: We treat some nonlinear problems.
\end{itemize}

\section{The Linear Problem}

The main results in this section are the following

\begin{theorem}
Assume (\ref{3}), (\ref{4}) then there exists a sequence $(u_{\epsilon
})_{0<\epsilon \leq 1}\subset W_{0}^{1,p}(\Omega )$ of weak solutions to (%
\ref{2}) and $u_{0}\in V_{p}$ such that $\ $ $\epsilon \nabla
_{X_{1}}u_{\epsilon }\rightarrow 0$ in $L^{p}(\Omega )$, $u_{\epsilon
}\rightarrow u_{0}$ in $V_{p}$ where $u_{0}$ satisfies (\ref{5}) for a.e $%
X_{1}\in \Omega ^{1}$.
\end{theorem}

\begin{corollary}
Assume (\ref{3}), (\ref{4}) then if $A$ is symmetric and continuous and $%
\partial \Omega \in C^{2}$, then there exists a unique $u_{0}\in V_{p}$ such
that $u_{0}(X_{1};\cdot )$ is the unique solution to (\ref{5}) in $%
W_{0}^{1,p}(\Omega _{X_{1}})$ for a.e $X_{1}$. Moreover the sequence $%
(u_{\epsilon })_{0<\epsilon \leq 1}$ of the unique solutions (in $%
W_{0}^{1,p}(\Omega )$) to (\ref{2}) converges in $V_{p}$ to $u_{0}$ and $%
\epsilon \nabla _{X_{1}}u_{\epsilon }\rightarrow 0$ in $L^{p}(\Omega )$.
\end{corollary}

\begin{proof}
This corollary follows immediately from Theorem 1 and uniqueness of the
solutions of (\ref{2}) and (\ref{5}) as mentioned in subsection 1.1 (Notice
that $\partial \Omega _{X_{1}}\in C^{2}$).
\end{proof}

\begin{theorem}
Assume (\ref{3}), (\ref{4}) then there exists a unique $u_{0}\in V_{p}$ such
that $u_{0}(X_{1},\cdot )$ is the unique entropy solution of (\ref{5}).
Moreover, the sequence of \ the entropy solutions $(u_{\epsilon
})_{0<\epsilon \leq 1}$ of (\ref{2}) converges to $u_{0}$ in $V_{p}$ and $%
\epsilon \nabla _{X_{1}}u_{\epsilon }\rightarrow 0$ in $L^{p}(\Omega )$.
\end{theorem}

\subsection{Weak convergence}

Let us prove the following primary result

\begin{theorem}
Assume (\ref{3}), (\ref{4}) then there exists a sequence $(u_{\epsilon
_{k}})_{k\in 
%TCIMACRO{\U{2115} }%
%BeginExpansion
\mathbb{N}
%EndExpansion
}\subset W_{0}^{1,p}(\Omega )$ of weak solutions to (\ref{2}) ($\epsilon
_{k}\rightarrow 0$ as $k\rightarrow \infty $) and $u_{0}\in V_{p}$ such that 
$\nabla _{X_{2}}u_{\epsilon _{k}}\rightharpoonup \nabla _{X_{2}}u_{0}$, $%
\epsilon _{k}\nabla _{X_{1}}u_{\epsilon _{k}}^{n}\rightharpoonup 0$, $%
u_{\epsilon _{k}}\rightharpoonup u_{0}$ in $L^{p}(\Omega -weak$. and $u_{0}$
satisfies (\ref{5}) for a.e $X_{1}\in \Omega ^{1}.$
\end{theorem}

\begin{proof}
By density let $(f_{n})_{n\in 
%TCIMACRO{\U{2115} }%
%BeginExpansion
\mathbb{N}
%EndExpansion
}\subset L^{2}(\Omega )$ be a sequence such that $f_{n}\rightarrow f$ in $%
L^{p}(\Omega )$, we can suppose that $\forall n\in $ $%
%TCIMACRO{\U{2115} }%
%BeginExpansion
\mathbb{N}
%EndExpansion
$ :$\left\Vert f_{n}\right\Vert _{L^{p}}\leq M$, $M\geq 0$. Consider the
regularized problem%
\begin{equation}
u_{\epsilon }^{n}\in H_{0}^{1}(\Omega ),\text{ \ }\int_{\Omega }A_{\epsilon
}\nabla u_{\epsilon }^{n}\cdot \nabla \varphi dx=\int_{\Omega }f_{n}\varphi
dx\ ,\ \varphi \in \mathcal{D(}\Omega \mathcal{)}  \label{7}
\end{equation}

Assumptions (\ref{2}) and (\ref{3}) shows that $u_{\epsilon }^{n}$ exists
and it is unique by the Lax-Milgram theorem. (Notice that $u_{\epsilon }^{n}$
also belongs to $W_{0}^{1,p}(\Omega )$). We introduce the function 
\begin{equation*}
\theta (t)=\dint\limits_{0}^{t}(1+\left\vert s\right\vert )^{p-2}ds\text{, }%
t\in 
%TCIMACRO{\U{211d}}%
%BeginExpansion
\mathbb{R}%
%EndExpansion
\end{equation*}%
This kind of function has been used in \cite{6}. We have $\theta ^{\prime
}(t)=$ $(1+\left\vert t\right\vert )^{p-2}\leq 1$ and $\theta (0)=0$,
therefore we have $\theta (u)\in H_{0}^{1}(\Omega )$ for every $u\in
H_{0}^{1}(\Omega )$. Testing with $\theta (u_{\epsilon }^{n})$ in (\ref{7})
and using the ellipticity assumption we deduce 
\begin{multline*}
\lambda \epsilon ^{2}\int_{\Omega }(1+\left\vert u_{\epsilon
}^{n}\right\vert )^{p-2}\left\vert \nabla _{X_{1}}u_{\epsilon
}^{n}\right\vert ^{2}dx+\lambda \int_{\Omega }(1+\left\vert u_{\epsilon
}^{n}\right\vert )^{p-2}\left\vert \nabla _{X_{2}}u_{\epsilon
}^{n}\right\vert ^{2}dx \\
\leq \int_{\Omega }f_{n}\theta (u_{\epsilon }^{n})dx\leq \frac{2}{p-1}%
\int_{\Omega }\left\vert f_{n}\right\vert (1+\left\vert u_{\epsilon
}^{n}\right\vert )^{p-1}dx\text{,}
\end{multline*}%
where we have used $\left\vert \theta (t)\right\vert \leq \frac{%
2(1+\left\vert t\right\vert )^{p-1}}{p-1}$. In the other hand, by H\"{o}%
lder's inequality we have%
\begin{equation*}
\int_{\Omega }\left\vert \nabla _{X_{2}}u_{\epsilon }^{n}\right\vert
^{p}dx\leq \left( \int_{\Omega }(1+\left\vert u_{\epsilon }^{n}\right\vert
)^{p-2}\left\vert \nabla _{X_{2}}u_{\epsilon }^{n}\right\vert ^{2}dx\right)
^{\frac{p}{2}}\left( \int_{\Omega }(1+\left\vert u_{\epsilon
}^{n}\right\vert )^{p}dx\right) ^{1-\frac{p}{2}}
\end{equation*}%
From the two previous integral inequalities we deduce 
\begin{multline*}
\int_{\Omega }\left\vert \nabla _{X_{2}}u_{\epsilon }^{n}\right\vert
^{p}dx\leq \left( \frac{2}{\lambda (p-1)}\int_{\Omega }\left\vert
f_{n}\right\vert (1+\left\vert u_{\epsilon _{k}}^{n}\right\vert
)^{p-1}dx\right) ^{\frac{p}{2}}\times \\
\left( \int_{\Omega }(1+\left\vert u_{\epsilon _{k}}^{n}\right\vert
)^{p}dx\right) ^{1-\frac{p}{2}}
\end{multline*}%
By H\"{o}lder's inequality we get 
\begin{equation}
\left\Vert \nabla _{X_{2}}u_{\epsilon }^{n}\right\Vert _{L^{p}(\Omega )}\leq
\left( \dfrac{2\left\Vert f_{n}\right\Vert _{L^{p}}}{\lambda (p-1)}\right) ^{%
\frac{1}{2}}\left( \int_{\Omega }(1+\left\vert u_{\epsilon }^{n}\right\vert
)^{p}dx\right) ^{\frac{1}{2p}}  \label{8}
\end{equation}%
Using Minkowki inequality we get 
\begin{equation*}
\left\Vert \nabla _{X_{2}}u_{\epsilon }^{n}\right\Vert _{L^{p}(\Omega
)}^{2}\leq C(1+\left\Vert u_{\epsilon }^{n}\right\Vert _{L^{p}(\Omega )}),
\end{equation*}%
Thanks to Poincar\'{e}'s inequality $\left\Vert u_{\epsilon }^{n}\right\Vert
_{L^{p}(\Omega )}\leq C_{\Omega }\left\Vert \nabla _{X_{2}}u_{\epsilon
}^{n}\right\Vert _{L^{p}(\Omega )}$ we obtain\qquad 
\begin{equation*}
\left\Vert \nabla _{X_{2}}u_{\epsilon }^{n}\right\Vert _{L^{p}(\Omega
)}^{2}\leq C^{\prime }(1+\left\Vert \nabla _{X_{2}}u_{\epsilon
}^{n}\right\Vert _{L^{p}(\Omega )}),
\end{equation*}

where the constant $C^{\prime }$ depends on $p$, $\lambda $, $mes(\Omega )$, 
$M$ and $C_{\Omega }$. Whence, we deduce%
\begin{equation}
\left\Vert u_{\epsilon }^{n}\right\Vert _{L^{p}(\Omega )}\text{, }\left\Vert
\nabla _{X_{2}}u_{\epsilon }^{n}\right\Vert _{L^{p}(\Omega )}\leq C^{\prime
\prime }  \label{9}
\end{equation}

Similarly we obtain 
\begin{equation}
\left\Vert \epsilon \nabla _{X_{1}}u_{\epsilon }^{n}\right\Vert
_{L^{p}(\Omega )}\leq C^{\prime \prime \prime },  \label{10}
\end{equation}%
where the constants $C^{\prime \prime }$, $C^{\prime \prime \prime }$ are
independent of $n$ and $\epsilon $, so 
\begin{equation}
\left\Vert u_{\epsilon }^{n}\right\Vert _{W^{1,p}(\Omega )}\leq \frac{Const}{%
\epsilon }  \label{11}
\end{equation}

Fix $\epsilon $, since $W^{1,p}(\Omega )$ is reflexive then (\ref{11})
implies that there exists a subsequence $(u_{\epsilon _{k}}^{n_{l}(\epsilon
)})_{l\in 
%TCIMACRO{\U{2115} }%
%BeginExpansion
\mathbb{N}
%EndExpansion
}$ and $u_{\epsilon }\in W_{0}^{1,p}(\Omega )$ such that $u_{\epsilon
}^{n_{l}(\epsilon )}\rightharpoonup u_{\epsilon }\in W_{0}^{1,p}(\Omega )$
(as $l\rightarrow \infty $) in $W^{1,p}(\Omega )-$weak. Now, passing to the
limit in (\ref{7}) as $l\rightarrow \infty $ we deduce%
\begin{equation}
\text{ \ }\int_{\Omega }A_{\epsilon }\nabla u_{\epsilon }\cdot \nabla
\varphi dx=\int_{\Omega }f\varphi dx\ ,\ \varphi \in \mathcal{D(}\Omega 
\mathcal{)}  \label{12}
\end{equation}

Whence $u_{\epsilon }$ is a weak solution of (\ref{2}) ($u_{\epsilon }=0$ on 
$\partial \Omega $ in the trace sense of $W^{1,p}-$functions, indeed the
trace operator is well defined since $\partial \Omega $ is Lipschitz).

Now, from (\ref{9}) and (\ref{10}) we deduce%
\begin{equation*}
\left\Vert u_{\epsilon }\right\Vert _{L^{p}(\Omega )}\leq \underset{%
l\rightarrow \infty }{\lim \inf }\left\Vert u_{\epsilon }^{n_{l}(\epsilon
)}\right\Vert _{L^{p}(\Omega )}\leq C^{\prime }
\end{equation*}

and similarly we obtain 
\begin{equation*}
\left\Vert \epsilon \nabla _{X_{1}}u_{\epsilon }\right\Vert _{L^{p}(\Omega )}%
\text{, }\left\Vert \nabla _{X_{2}}u_{\epsilon }\right\Vert _{L^{p}(\Omega
)}\leq C^{\prime }
\end{equation*}

Using reflexivity and continuity of the derivation operator on $\mathcal{D}%
^{\prime }(\Omega )$ one can extract a subsequence $(u_{\epsilon
_{k}})_{k\in 
%TCIMACRO{\U{2115} }%
%BeginExpansion
\mathbb{N}
%EndExpansion
}$ such that $\nabla _{X_{2}}u_{\epsilon _{k}}\rightharpoonup \nabla
_{X_{2}}u_{0}$, $\epsilon _{k}\nabla _{X_{1}}u_{\epsilon
_{k}}^{n}\rightharpoonup 0$, $u_{\epsilon _{k}}\rightharpoonup u_{0}$ in $%
L^{p}(\Omega )-weak$. Passing to the limit in (\ref{12}) we get 
\begin{equation}
\int_{\Omega }A_{22}\nabla _{X_{2}}u_{0}\cdot \nabla _{X_{2}}\varphi
dx=\int_{\Omega }f\varphi dx\ ,\ \varphi \in \mathcal{D(}\Omega \mathcal{)}
\label{13}
\end{equation}

Now, we will prove that $u_{0}\in V_{p}.$ Since $\nabla _{X_{2}}u_{\epsilon
_{k}}\rightharpoonup \nabla _{X_{2}}u_{0}$ and $u_{\epsilon
_{k}}\rightharpoonup u_{0}$ in $L^{p}(\Omega )-weak$ then there exists a
sequence $(U_{n})_{n\in 
%TCIMACRO{\U{2115} }%
%BeginExpansion
\mathbb{N}
%EndExpansion
}\subset conv(\left\{ u_{\epsilon _{k}}\right\} _{k\in 
%TCIMACRO{\U{2115} }%
%BeginExpansion
\mathbb{N}
%EndExpansion
})$ such that $\nabla _{X_{2}}U_{n}\rightarrow \nabla _{X_{2}}u_{0}$ in $%
L^{p}(\Omega )-strong$, where $conv(\left\{ u_{\epsilon _{k}}\right\} _{k\in 
%TCIMACRO{\U{2115} }%
%BeginExpansion
\mathbb{N}
%EndExpansion
})$ is the convex hull of the set $\left\{ u_{\epsilon _{k}}\right\} _{k\in 
%TCIMACRO{\U{2115} }%
%BeginExpansion
\mathbb{N}
%EndExpansion
}$. Notice that we have $U_{n}\in W_{0}^{1,p}(\Omega )$ then -up to a
subsequence- we have $U_{n}(X_{1},\cdot )\in W_{0}^{1,p}(\Omega _{X_{1}})$,
a.e $X_{1}\in \Omega ^{1}$. And we also have -up to a subsequence- $\nabla
_{X_{2}}U_{n}(X_{1},\cdot )\rightarrow $ $\nabla _{X_{2}}u_{0}(X_{1},.)$ in $%
L^{p}(\Omega _{X_{1}})-strong$ a.e $X_{1}\in \Omega ^{1}$. Whence $%
u_{0}(X_{1},.)\in W_{0}^{1,p}(\Omega _{X_{1}})$ for a.e $X_{1}\in \Omega
^{1} $, so $u_{0}\in V_{p}$.

Finally, we will prove that $u_{0}$ is a solution of (\ref{5}). Let $E$ be a
Banach space, a family of vectors $\left\{ e_{n}\right\} _{n\in 
%TCIMACRO{\U{2115} }%
%BeginExpansion
\mathbb{N}
%EndExpansion
}$ in $E$ is said to be a Banach basis or a Schauder basis of $E$ if for
every $x\in E$ there exists a family of scalars $(\alpha _{n})_{n\in 
%TCIMACRO{\U{2115} }%
%BeginExpansion
\mathbb{N}
%EndExpansion
}$ such that $x=\dsum\limits_{n=0}^{\infty }\alpha _{n}e_{n}$, where the
series converges in the norm of $E$. Notice that Schauder basis does not
always exist. In \cite{5} P. Enflo has constructed a separable reflexive
Banach space without Schauder basis!. However, the Sobolev space $%
W_{0}^{1,r} $ ( $1<r<\infty $) has a Schauder basis whenever the boundary of
the domain is sufficiently smooth \cite{7}. Now, we are ready to finish the
proof. Let $(U_{i}\times V_{i})_{i\in 
%TCIMACRO{\U{2115} }%
%BeginExpansion
\mathbb{N}
%EndExpansion
}$ be a countable covering of $\Omega $ such that $U_{i}\times V_{i}\subset
\Omega $ where $U_{i}\subset 
%TCIMACRO{\U{211d} }%
%BeginExpansion
\mathbb{R}
%EndExpansion
^{q},V_{i}\subset 
%TCIMACRO{\U{211d} }%
%BeginExpansion
\mathbb{R}
%EndExpansion
^{N-q}$ are two bounded open domains, where $\partial V_{i}$ is smooth ($%
V_{i}$ are Euclidian balls for example), such a covering always exists. Now,
fix $\psi \in \mathcal{D}(V_{i})$ then it follows from (\ref{13}) that for
every $\varphi \in \mathcal{D}(U_{i})$ we have%
\begin{equation*}
\int_{U_{i}}\varphi dX_{1}\int_{V_{i}}A_{22}\nabla _{X_{2}}u_{0}\cdot \nabla
_{X_{2}}\psi dX_{2}\ =\int_{U_{i}}\varphi dX_{1}\int_{V_{i}}f\psi dX_{2}\ 
\end{equation*}

Whence for a.e $X_{1}\in U_{i}$ we have 
\begin{equation*}
\int_{V_{i}}A_{22}(X_{1},\cdot )\nabla _{X_{2}}u_{0}(X_{1},\cdot )\cdot
\nabla _{X_{2}}\psi dX_{2}\ =\int_{V_{i}}f(X_{1},\cdot )\psi dX_{2}
\end{equation*}

Notice that by density we can take $\psi \in W_{0}^{1,p^{\prime }}(V_{i})$
where $p^{\prime }$ is the conjugate of $p$. Using the same techniques as in 
\cite{3}, where we use a Schauder basis of $W_{0}^{1,p^{\prime }}(V_{i})$
and a partition of the unity, one can easily obtain%
\begin{equation*}
\int_{\Omega _{X_{1}}}A_{22}(X_{1},\cdot )\nabla _{X_{2}}u_{0}(X_{1},\cdot
)\cdot \nabla _{X_{2}}\varphi dx=\int_{\Omega _{X_{1}}}f(X_{1},\cdot
)\varphi dx\text{, }\varphi \in \mathcal{D(}\Omega \mathcal{)}\text{,}
\end{equation*}

for a.e $X_{1}\in \Omega ^{1}$. Finally, since $u_{0}(X_{1},\cdot )\in
W_{0}^{1,p}(\Omega _{X_{1}})$ (as proved above) then $u_{0}(X_{1},\cdot )$
is a solution of (\ref{5}) (Notice that $\Omega _{X_{1}}$ is also a
Lipschitz domain so the trace operator is well defined).
\end{proof}

\subsection{Strong convergence}

Theorem 1 will be proved in three steps. the proof is based on the use of
the approximated problem (\ref{7}). In the first step, we shall construct
the solution of the limit problem

$\mathbf{Step1:}$ Let $u_{\epsilon }^{n}\in H_{0}^{1}(\Omega )$ be the
unique solution to (\ref{7}), existence and uniqueness of $u_{\epsilon }^{n}$
follows from assumptions (\ref{3}), (\ref{4}) as mentioned previously. One
have the following

\begin{proposition}
Assume (\ref{3}), (\ref{4}) then there exists $(u_{0}^{n})_{n\in 
%TCIMACRO{\U{2115} }%
%BeginExpansion
\mathbb{N}
%EndExpansion
}\subset V_{2}$ such that $\epsilon u_{\epsilon }^{n}\rightarrow 0$ in $%
L^{2}(\Omega )$, $u_{\epsilon }^{n}\rightarrow u_{0}^{n}$ in $V_{2}$ for
every $n\in 
%TCIMACRO{\U{2115} }%
%BeginExpansion
\mathbb{N}
%EndExpansion
$, in particular the two convergences holds in $L^{p}(\Omega )$ and $V_{p}$
respectively. And $u_{0}^{n}$ is the unique weak solution in $V_{2}$ to the
problem%
\begin{equation}
\left\{ 
\begin{array}{cc}
\func{div}_{X_{2}}(A_{22}(X_{1},\cdot )\nabla _{X_{2}}u_{0}^{n}(X_{1},\cdot
))=f_{n}(X_{1},\cdot )\text{, }X_{1}\in \Omega ^{1} &  \\ 
u_{0}^{n}(X_{1},\cdot )=0\text{ on }\partial \Omega _{X_{1}}\text{\ \ \ \ \
\ \ \ \ \ \ \ \ \ \ \ \ \ \ \ \ \ \ \ \ \ \ \ \ \ \ \ \ \ \ \ \ } & \text{ }%
\end{array}%
\right.  \label{14}
\end{equation}
\end{proposition}

\begin{proof}
This result follows from the $L^{2}-$theory (Theorem 1 in \cite{3}), The
convergences in $V_{p}$ and $L^{p}(\Omega )$ follow from the continuous
embedding $V_{2}\hookrightarrow V_{p}$, $L^{2}(\Omega )\hookrightarrow
L^{p}(\Omega )$ ($p<2$).
\end{proof}

Now, we construct $u_{0}$ the solution of the limit problem (\ref{5}).
Testing with $\varphi =\theta (u_{0}^{n}(X_{1},\cdot ))$ in the weak
formulation of (\ref{14}) ($\theta $ is the function introduced in
subsection 2.1) and estimating like in the proof of Theorem 3 we obtain as
in (\ref{8})%
\begin{eqnarray}
&&\left\Vert \nabla _{X_{2}}u_{0}^{n}(X_{1},\cdot )\right\Vert
_{L^{p}(\Omega _{X_{1}})}  \notag \\
&\leq &\left( \dfrac{\left\Vert f_{n}(X_{1},\cdot )\right\Vert
_{L^{p}(\Omega _{X_{1}})}}{\lambda (p-1)}\right) ^{\frac{1}{2}}\times \left(
\int_{\Omega _{X_{1}}}(1+\left\vert u_{0}^{n}(X_{1},\cdot )\right\vert
)^{p}dX_{2}\right) ^{\frac{1}{2p}}  \label{15}
\end{eqnarray}

Integrating over $\Omega ^{1}$ and using Cauchy-Schwaz's inequality in the
right hand side we get 
\begin{equation*}
\left\Vert \nabla _{X_{2}}u_{0}^{n}\right\Vert _{L^{p}(\Omega )}^{p}\leq
C\left\Vert f_{n}\right\Vert _{L^{p}(\Omega )}^{\frac{p}{2}}\left(
\int_{\Omega }(1+\left\vert u_{0}^{n}\right\vert )^{p}dx\right) ^{\frac{1}{2}%
}
\end{equation*}

and therefore 
\begin{equation*}
\left\Vert \nabla _{X_{2}}u_{0}^{n}\right\Vert _{L^{p}(\Omega )}^{2}\leq
C^{\prime }(1+\left\Vert u_{0}^{n}\right\Vert _{L^{p}(\Omega )})
\end{equation*}

Using Poincar\'{e}'s inequality $\left\Vert u_{0}^{n}\right\Vert
_{L^{p}(\Omega )}\leq C_{\Omega }\left\Vert \nabla
_{X_{2}}u_{0}^{n}\right\Vert _{L^{p}(\Omega )}$ ( which holds since $%
u_{0}^{n}(X_{1},\cdot )\in W_{0}^{1,p}(\Omega _{X_{1}})$ a.e $X_{1}\in
\Omega ^{1}$), one can obtain the estimate%
\begin{equation}
\left\Vert u_{0}^{n}\right\Vert _{L^{p}(\Omega )}\leq C^{\prime \prime }%
\text{ for every }n\in 
%TCIMACRO{\U{2115} }%
%BeginExpansion
\mathbb{N}
%EndExpansion
\text{, }  \label{16}
\end{equation}

where $C^{\prime \prime }$ is independent of $n$. Now, using the linearity
of the problem and (\ref{14}) with the test function $\theta
(u_{0}^{n}(X_{1},\cdot )-u_{0}^{m}(X_{1},\cdot ))$, $m,n\in 
%TCIMACRO{\U{2115} }%
%BeginExpansion
\mathbb{N}
%EndExpansion
$ one can obtain like in (\ref{15})%
\begin{multline*}
\left\Vert \nabla _{X_{2}}\left( u_{0}^{n}(X_{1},\cdot
)-u_{0}^{m}(X_{1},\cdot )\right) \right\Vert _{L^{p}(\Omega _{X_{1}})} \\
\leq \left( \dfrac{\left\Vert f_{n}(X_{1},\cdot )-f_{m}(X_{1},\cdot
)\right\Vert _{L^{p}(\Omega _{X_{1}})}}{\lambda (p-1)}\right) ^{\frac{1}{2}%
}\times \\
\left( \int_{\Omega _{X_{1}}}(1+\left\vert u_{0}^{n}(X_{1},\cdot
)-u_{0}^{m}(X_{1},\cdot )\right\vert )^{p}dX_{2}\right) ^{\frac{1}{2p}}
\end{multline*}

\bigskip integrating over $\Omega ^{1}$ and using Cauchy-Schwarz and (\ref%
{16}) yields%
\begin{equation*}
\left\Vert \nabla _{X_{2}}(u_{0}^{n}-u_{0}^{m})\right\Vert _{L^{p}(\Omega
)}\leq C\left\Vert f_{n}-f_{m}\right\Vert _{L^{p}(\Omega )}^{\frac{1}{2}},
\end{equation*}

where $C$ is independent of $m$ and $n$. The Poincar\'{e}'s inequality shows
that 
\begin{equation*}
\left\Vert u_{0}^{n}-u_{0}^{m}\right\Vert _{V_{p}}\leq C^{\prime }\left\Vert
f_{n}-f_{m}\right\Vert _{L^{p}(\Omega )}^{\frac{1}{2}}
\end{equation*}

Since $(f_{n})_{n\in 
%TCIMACRO{\U{2115} }%
%BeginExpansion
\mathbb{N}
%EndExpansion
}$ is a converging sequence in $L^{p}(\Omega )$ then this last inequality
shows that $(u_{0}^{n})_{n\in 
%TCIMACRO{\U{2115} }%
%BeginExpansion
\mathbb{N}
%EndExpansion
}$ is a Cauchy sequence in $V_{p}$, consequently there exists $u_{0}\in
V_{p} $ such that $u_{0}^{n}\rightarrow u_{0}$ in $V_{p}$. Now, passing to
the limit in (\ref{7}) as $\epsilon \rightarrow 0$ we get 
\begin{equation*}
\dint\limits_{\Omega }A_{22}\nabla _{X_{2}}u_{0}^{n}\cdot \nabla
_{X_{2}}\varphi dX_{2}=\dint\limits_{\Omega }f_{n}\varphi dX_{2}\text{, }%
\varphi \in \mathcal{D}(\Omega )
\end{equation*}

Passing to the limit as $n\rightarrow \infty $ we deduce 
\begin{equation*}
\int_{\Omega }A_{22}\nabla _{X_{2}}u_{0}\cdot \nabla _{X_{2}}\varphi
dX_{2}=\int_{\Omega }f\varphi dX_{2}\text{, }\varphi \in \mathcal{D}(\Omega )
\end{equation*}

\bigskip Then it follows as proved in Theorem 3 that $u_{0}$ satisfies (\ref%
{5}). Whence we have proved the following

\begin{proposition}
Under assumption of Proposition 1 there exists $u_{0}\in V_{p}$ solution to (%
\ref{5}) such that $u_{0}^{n}\rightarrow u_{0}$ in $V_{p}$ where $%
(u_{0}^{n})_{n\in 
%TCIMACRO{\U{2115} }%
%BeginExpansion
\mathbb{N}
%EndExpansion
}$ is the sequence given in Proposition 1
\end{proposition}

$\mathbf{Step2:}$ In this second step we will construct the sequence $%
(u_{\epsilon })_{0<\epsilon \leq 1}$ solutions of (\ref{2}), one can prove
the following

\begin{proposition}
There exists a sequence $(u_{\epsilon })_{0<\epsilon \leq 1}\subset
W_{0}^{1,p}(\Omega )$ of weak solutions to (\ref{2}) such that $u_{\epsilon
}^{n}\rightarrow u_{\epsilon }$ in $W^{1,p}(\Omega )$ for every $\epsilon $
fixed. Moreover, $u_{\epsilon }^{n}\rightarrow u_{\epsilon }$ in $V_{p}$ and 
$\epsilon \nabla _{X_{2}}u_{\epsilon }^{n}\rightarrow \epsilon \nabla
_{X_{2}}u_{\epsilon }$, uniformly in $\epsilon $.
\end{proposition}

\begin{proof}
Using the linearity of (\ref{7}) testing with $\theta (u_{\epsilon
}^{n}-u_{\epsilon }^{m})$, $m,n\in 
%TCIMACRO{\U{2115} }%
%BeginExpansion
\mathbb{N}
%EndExpansion
$ we obtain as in (\ref{8}) 
\begin{equation*}
\left\Vert \nabla _{X_{2}}u_{\epsilon }^{n}-u_{\epsilon }^{m}\right\Vert
_{L^{p}(\Omega )}\leq \left( \dfrac{\left\Vert f_{n}-f_{m}\right\Vert
_{L^{p}}}{\lambda (p-1)}\right) ^{\frac{1}{2}}\left( \int_{\Omega
}(1+\left\vert u_{\epsilon }^{n}-u_{\epsilon }^{m}\right\vert )^{p}\right) ^{%
\frac{1}{2p}}
\end{equation*}

And (\ref{9}) gives 
\begin{equation*}
\left\Vert \nabla _{X_{2}}(u_{\epsilon }^{n}-u_{\epsilon }^{m})\right\Vert
_{L^{p}(\Omega )}\leq C\left\Vert f_{n}-f_{m}\right\Vert _{L^{p}}^{\frac{1}{2%
}}
\end{equation*}

where $C$ is independent of $\epsilon $ and $n$, whence Poincar\'{e}'s
inequality implies 
\begin{equation}
\left\Vert u_{\epsilon }^{n}-u_{\epsilon }^{m}\right\Vert _{V_{p}}\leq
C^{\prime }\left\Vert f_{n}-f_{m}\right\Vert _{L^{p}}^{\frac{1}{2}}
\label{17}
\end{equation}

Similarly we obtain 
\begin{equation}
\left\Vert \epsilon \nabla _{X_{2}}(u_{\epsilon }^{n}-u_{\epsilon
}^{m})\right\Vert _{L^{p}(\Omega )}\leq C^{\prime \prime }\left\Vert
f_{n}-f_{m}\right\Vert _{L^{p}}^{\frac{1}{2}}  \label{18}
\end{equation}

its follows that 
\begin{equation*}
\left\Vert u_{\epsilon }^{n}-u_{\epsilon }^{m}\right\Vert _{W^{1,p}(\Omega
)}\leq \frac{C}{\epsilon }\left\Vert f_{n}-f_{m}\right\Vert _{L^{p}}^{\frac{1%
}{2}}
\end{equation*}

The last inequality implies that for every $\epsilon $ fixed $(u_{\epsilon
}^{n})_{n\in 
%TCIMACRO{\U{2115} }%
%BeginExpansion
\mathbb{N}
%EndExpansion
}$ is a Cauchy sequence in $W_{0}^{1,p}(\Omega )$, Then there exists $%
u_{\epsilon }\in W_{0}^{1,p}(\Omega )$ such that $u_{\epsilon
}^{n}\rightarrow u_{\epsilon }$ in $W^{1,p}(\Omega )$, then the passage to
the limit in (\ref{7}) shows that $u_{\epsilon }$ is a weak solution of (\ref%
{2}). Finally (\ref{17}) and (\ref{18}) show that $u_{\epsilon
}^{n}\rightarrow u_{\epsilon }$ (resp $\epsilon \nabla _{X_{2}}u_{\epsilon
}^{n}\rightarrow \epsilon \nabla _{X_{2}}u_{\epsilon }$) in $V_{p}$ ( resp
in $L^{p}(\Omega )$) uniformly in $\epsilon $.
\end{proof}

$\mathbf{Step3:}$ Now, we are ready to conclude. Proposition 1, 2 and 3
combined with the triangular inequality show that $u_{\epsilon }\rightarrow
u_{0}$ in $V_{p}$ and $\epsilon \nabla _{X_{2}}u_{\epsilon }\rightarrow 0$
in $L^{p}(\Omega )$, and the proof of Theorem 1 is finished.

\subsection{Convergence of the entropy solutions}

As mentioned in section 1 the entropy solution $u_{\epsilon }$ of (\ref{2})
exists and it is unique. We shall construct this entropy solution. Using the
approximated problem (\ref{7}), one has a $W^{1,p}-$strongly converging
sequence $u_{\epsilon }^{n}\rightarrow u_{\epsilon }\in W_{0}^{1,p}(\Omega )$
as shown in Proposition 3. We will show that $u_{\epsilon }\in \mathcal{T}%
_{0}^{1,2}(\Omega )$. Clearly we have$T_{k}(u_{\epsilon }^{n})\in
H_{0}^{1}(\Omega )$ for every $k>0$. Now testing with $T_{k}(u_{\epsilon
}^{n})$ in (\ref{7}) we obtain 
\begin{equation*}
\int_{\Omega }A_{\epsilon }\nabla u_{\epsilon }^{n}\cdot \nabla
T_{k}(u_{\epsilon }^{n})dx=\int_{\Omega }f_{n}T_{k}(u_{\epsilon }^{n})dx
\end{equation*}

Using the ellipticity assumption we get%
\begin{equation}
\int_{\Omega }\left\vert \nabla T_{k}(u_{\epsilon }^{n})\right\vert ^{2}\leq 
\frac{Mk}{\lambda (1+\epsilon ^{2})}  \label{19}
\end{equation}

Fix $\epsilon ,k$, we have $u_{\epsilon }^{n}\rightarrow u_{\epsilon }$ in $%
L^{p}(\Omega )$ then there exists a subsequence $(u_{\epsilon
}^{n_{l}})_{l\in 
%TCIMACRO{\U{2115} }%
%BeginExpansion
\mathbb{N}
%EndExpansion
}$ such that $u_{\epsilon }^{n_{l}}\rightarrow u_{\epsilon }$ a.e $x\in
\Omega $ and since $T_{k}$ is bounded then it follows that $%
T_{k}(u_{\epsilon }^{n_{l}})\rightarrow T_{k}(u_{\epsilon })$ a.e in $\Omega 
$ and strongly in $L^{2}(\Omega )$ whence $u_{\epsilon }\in \mathcal{T}%
_{0}^{1,2}(\Omega )$.

It follows by (\ref{19}) that there exists a subsequence still labelled $%
T_{k}(u_{\epsilon }^{n_{l}})$ such that $\nabla T_{k}(u_{\epsilon
}^{n_{l}})\rightarrow v_{\epsilon ,k}\in L^{2}(\Omega )$.The continuity of $%
\nabla $ on $\mathcal{D}^{\prime }\mathcal{(}\Omega \mathcal{)}$ implies
that $v_{\epsilon ,k}=\nabla T_{k}(u_{\epsilon })$, whence $%
T_{k}(u_{\epsilon }^{n_{l}})\rightarrow T_{k}(u_{\epsilon })$ in $%
H^{1}(\Omega )$. Now, since $T_{k}(u_{\epsilon }^{n_{l}})\in
H_{0}^{1}(\Omega )$ then we deduce that $T_{k}(u_{\epsilon })\in
H_{0}^{1}(\Omega )$.

It follows \cite{2} that 
\begin{equation*}
\int_{\Omega }A_{\epsilon }\nabla u_{\epsilon }\cdot \nabla
T_{k}(u_{\epsilon }-\varphi )dx\leq \int_{\Omega }fT_{k}(u_{\epsilon
}-\varphi )dx
\end{equation*}%
Whence $u_{\epsilon }$ is the entropy solution of (\ref{2}). Similarly the
function $u_{0}$ (constructed in Proposition 2) is the entropy solution to (%
\ref{5}) for a.e $X_{1}$ The uniqueness of $u_{0}$ in $V_{p}$ follows from
the uniqueness of the entropy solution of problem (\ref{5}). Finally, the
convergences given in Theorem 2 follows from Theorem 1.

\begin{remark}
Uniqueness of the entropy solutions implies that it does not depend on the
choice of the approximated sequence $(f_{n})_{n}$.
\end{remark}

\subsection{A regularity result for the entropy solution of the limit problem%
}

In this subsection we assume that $\Omega =\omega _{1}\times \omega _{2}$
where $\omega _{1},$ $\omega _{2}$ are two bounded Lipschitz domains of $%
%TCIMACRO{\U{211d} }%
%BeginExpansion
\mathbb{R}
%EndExpansion
^{q}$, $%
%TCIMACRO{\U{211d} }%
%BeginExpansion
\mathbb{R}
%EndExpansion
^{N-q}$ respectively. We introduce the space 
\begin{equation*}
W_{p}=\left\{ u\in L^{p}(\Omega )\mid \nabla _{X_{1}}u\in L^{p}(\Omega
)\right\}
\end{equation*}
We suppose the following 
\begin{equation}
f\in W_{p}\text{ and }A_{22}(x)=A_{22}(X_{2})\text{ i.e }A_{22}\text{ is
independent of }X_{1}  \label{21}
\end{equation}

\begin{theorem}
Assume (\ref{3}), (\ref{4}), (\ref{21}) then $u_{0}\in W^{1,p}(\Omega )$,
where $u_{0}$ is the entropy solution of (\ref{5}).
\end{theorem}

\begin{proof}
Let $(u_{0}^{n})$ the sequence constructed in subsection 2.2, we have $%
u_{0}^{n}\rightarrow u_{0}$ in $V_{p}$, where $u_{0}$ is the entropy
solution of (\ref{5}) as mentioned in the above subsection.

Let $\omega _{1}^{\prime }\subset \subset \omega _{1}$ be an open subset$,$
for $0<h<d(\partial \omega _{1},\omega _{1}^{\prime })$ and for $X_{1}\in
\omega _{1}^{\prime }$ we set $\tau
_{h}^{i}u_{0}^{n}=u_{0}^{n}(X_{1}+he_{i},X_{2})$ where $e_{i}=(0,..,1,..,0)$
then we have by (\ref{14}) 
\begin{equation*}
\int_{\omega _{2}}A_{22}\nabla _{X_{2}}(\tau
_{h}^{i}u_{0}^{n}-u_{0}^{n})\cdot \nabla _{X_{2}}\varphi dX_{2}=\int_{\omega
_{2}}(\tau _{h}^{i}f_{n}-f_{n})\varphi dX_{2}\text{ ,\ \ }\varphi \in 
\mathcal{D}(\omega _{2})
\end{equation*}

where we have used $A_{22}(x)=A_{22}(X_{2}).$

We introduce the function $\theta _{\delta }(t)=\dint\limits_{0}^{t}\left(
\delta +\left\vert s\right\vert \right) ^{p-2}ds$, $\delta >0$, $t\in 
%TCIMACRO{\U{211d} }%
%BeginExpansion
\mathbb{R}
%EndExpansion
$ we have $0<\theta _{\delta }^{\prime }(t)=\left( \delta +\left\vert
t\right\vert \right) ^{p-2}\leq \delta ^{p-2}$ and $\left\vert \theta
_{\delta }(t)\right\vert \leq \frac{2(\delta +\left\vert t\right\vert )^{p-1}%
}{p-1}$

Testing with $\varphi =\frac{1}{h}\theta _{\delta }(\frac{\tau
_{h}^{i}u_{0}^{n}-u_{0}^{n}}{h})\in H_{0}^{1}(\omega _{2})$. To make the
notations less heavy we set%
\begin{equation*}
U=\frac{\tau _{h}^{i}u_{0}^{n}-u_{0}^{n}}{h}\text{, }\frac{\text{ }(\tau
_{h}^{i}f_{n}-f_{n})}{h}=F
\end{equation*}

Then we get 
\begin{equation*}
\int_{\omega _{2}}\theta _{\delta }^{\prime }(U)A_{22}\nabla _{X_{2}}U\cdot
\nabla _{X_{2}}UdX_{2}=\int_{\omega _{2}}F\theta _{\delta }(U)dX_{2}\text{ }
\end{equation*}

Using the ellipticity assumption for the left hand side and H\"{o}lder's
inequality for the right hand side of the previous inequality we deduce 
\begin{equation*}
\lambda \int_{\omega _{2}}\theta _{\delta }^{\prime }(U)\left\vert \nabla
_{X_{2}}U\right\vert ^{2}dX_{2}\leq \frac{2}{p-1}\left\Vert F\right\Vert
_{L^{p}(\omega _{2})}\left( \int_{\omega _{2}}\left( \delta +\left\vert
U\right\vert \right) ^{p}dX_{2}\right) ^{\frac{p-1}{p}}
\end{equation*}

Using H\"{o}lder's inequality we derive 
\begin{multline*}
\left\Vert \nabla _{X_{2}}U\right\Vert _{L^{p}(\omega _{2})}^{p}\leq \left(
\int_{\omega _{2}}\theta _{\delta }^{\prime }(U)\left\vert \nabla
_{X_{2}}U\right\vert ^{2}dX_{2}\right) ^{\frac{p}{2}}\left( \int_{\omega
_{2}}\theta _{\delta }^{\prime }(U)^{^{\frac{p}{p-2}}}dX_{2}\right) ^{\frac{%
2-p}{2}} \\
\leq \left( \frac{2}{\lambda (p-1)}\left\Vert F\right\Vert _{L^{p}(\omega
_{2})}\left( \int_{\omega _{2}}\left( \delta +\left\vert U\right\vert
\right) ^{p}dX_{2}\right) ^{\frac{p-1}{p}}\right) ^{\frac{p}{2}}\times \\
\left( \int_{\omega _{2}}\theta _{\delta }^{\prime }(U)^{^{\frac{p}{p-2}%
}}dX_{2}\right) ^{\frac{2-p}{2}}
\end{multline*}

Then we deduce%
\begin{equation*}
\left\Vert \nabla _{X_{2}}U\right\Vert _{L^{p}(\omega _{2})}^{2}\leq \frac{2%
}{\lambda (p-1)}\left\Vert F\right\Vert _{L^{p}(\omega _{2})}\left(
\int_{\omega _{2}}\left( \delta +\left\vert U\right\vert \right)
^{p}dX_{2}\right) ^{\frac{1}{p}}
\end{equation*}

Now passing to the limit as $\delta \rightarrow 0$ using the Lebesgue
theorem we deduce%
\begin{equation*}
\left\Vert \nabla _{X_{2}}U\right\Vert _{L^{p}(\omega _{2})}^{2}\leq \frac{2%
}{\lambda (p-1)}\left\Vert F\right\Vert _{L^{p}(\omega _{2})}\left(
\int_{\omega _{2}}\left( \left\vert U\right\vert \right) ^{p}dX_{2}\right) ^{%
\frac{1}{p}},
\end{equation*}

and Poincar\'{e}'s inequality gives 
\begin{equation*}
\left\Vert \nabla _{X_{2}}U\right\Vert _{L^{p}(\omega _{2})}\leq \frac{%
2C_{\omega _{2}}}{\lambda (p-1)}\left\Vert F\right\Vert _{L^{p}(\omega _{2})}
\end{equation*}

Now, integrating over $\omega _{1}^{\prime }$ yields 
\begin{equation*}
\left\Vert \frac{\tau _{h}^{i}u_{0}^{n}-u_{0}^{n}}{h}\right\Vert
_{L^{p}(\omega _{1}^{\prime }\times \omega _{2})}\leq \frac{2C_{\omega _{2}}%
}{\lambda (p-1)}\left\Vert \frac{\text{ }(\tau _{h}^{i}f_{n}-f_{n})}{h}%
\right\Vert _{L^{p}(\omega _{1}^{\prime }\times \omega _{2})}
\end{equation*}

Passing to the limit as $n\rightarrow \infty $ using the invariance of the
Lebesgue measure under translations we get 
\begin{equation*}
\left\Vert \frac{\tau _{h}^{i}u_{0}-u_{0}}{h}\right\Vert _{L^{p}(\omega
_{1}^{\prime }\times \omega _{2})}\leq \frac{2C_{\omega _{2}}}{\lambda (p-1)}%
\left\Vert \frac{\text{ }(\tau _{h}^{i}f-f)}{h}\right\Vert _{L^{p}(\omega
_{1}^{\prime }\times \omega _{2})}
\end{equation*}

Whence, since $f\in W_{p}$ then 
\begin{equation*}
\left\Vert \frac{\tau _{h}^{i}u_{0}-u_{0}}{h}\right\Vert _{L^{p}(\omega
_{1}^{\prime }\times \omega _{2})}\leq C,
\end{equation*}

where $C$ is independent of $h$, therefore we have $\nabla _{X_{1}}u_{0}\in
L^{p}(\Omega )$. Combining this with $u_{0}\in V_{p}$ we get the desired
result.
\end{proof}

\section{The Rate of convergence Theorem}

In this section we suppose that $\Omega =\omega _{1}\times \omega _{2}$
where $\omega _{1},\omega _{2}$ are two bounded Lipschitz domains of $%
%TCIMACRO{\U{211d} }%
%BeginExpansion
\mathbb{R}
%EndExpansion
^{q}$ and $%
%TCIMACRO{\U{211d} }%
%BeginExpansion
\mathbb{R}
%EndExpansion
^{N-q}$ respectively. We suppose that $A_{12}$, $A_{22}$ and $f$ depend on $%
X_{2}$ only i.e $A_{12}(x)=A_{12}(X_{2})$, $A_{22}(x)=A_{22}(X_{2})$ and $%
f(x)=f(X_{2})\in L^{p}(\omega _{2})$ ($1<p<2$), $f\notin L^{2}(\omega _{2})$.

Let $u_{\epsilon }$, $u_{0}$ be the unique entropy solutions of (\ref{2}), (%
\ref{5}) respectively then under the above assumptions we have the following

\begin{theorem}
For every $\omega _{1}^{\prime }\subset \subset \omega _{1}$ and $m\in 
%TCIMACRO{\U{2115} }%
%BeginExpansion
\mathbb{N}
%EndExpansion
^{\ast }$ there exists $C\geq 0$ independent of $\ \epsilon $ such that 
\begin{equation*}
\left\Vert u_{\epsilon }-u_{0}\right\Vert _{W^{p}(\omega _{1}^{\prime
}\times \omega _{2})}\leq C\epsilon ^{m}
\end{equation*}
\end{theorem}

\begin{proof}
Let $u_{\epsilon }$, $u_{0}$ be the entropy solutions of (\ref{2}), (\ref{5}%
) respectively, we use the approximated sequence $(u_{\epsilon
}^{n})_{\epsilon ,n}$, $(u_{0}^{n})_{n}$ introduced in section 2.
Subtracting (\ref{14}) from (\ref{7}) we obtain%
\begin{equation*}
\int_{\Omega }A_{\epsilon }\nabla (u_{\epsilon }^{n}-u_{0}^{n})\cdot \nabla
\varphi dx=0,
\end{equation*}

where we have used that $u_{0}^{n}$ is independent of $X_{1}$ (since $f$ and 
$A_{22}$ are independent of $X_{1}$) and that $A_{12}$ is independent of $%
X_{1}$.

Let $\omega _{1}^{\prime }\subset \subset \omega _{1}$ then there exists $%
\omega _{1}^{\prime }\subset \subset \omega _{1}^{\prime \prime }\subset
\subset \omega _{1}$. We introduce the function $\rho \in \mathcal{D}(\omega
_{1})$ such that $Supp(\rho )\subset \omega _{1}^{\prime \prime }$ and\ $%
\rho =1$ on $\omega _{1}^{\prime }$\textbf{( }we can choose $0\leq \rho \leq
1$) Testing with $\varphi =\rho ^{2}\theta _{\delta }(u_{\epsilon
}^{n}-u_{0}^{n})\in H_{0}^{1}(\Omega )$ (we can check easily that this
function belongs to $H_{0}^{1}(\Omega )$ using approximation argument) in
the above integral equality we get%
\begin{multline*}
\int_{\Omega }\rho ^{2}\theta _{\delta }^{\prime }(u_{\epsilon
}^{n}-u_{0}^{n})A_{\epsilon }\nabla (u_{\epsilon }^{n}-u_{0}^{n})\cdot
\nabla (u_{\epsilon }^{n}-u_{0}^{n})dx \\
=-\int_{\Omega }\rho \theta _{\delta }(u_{\epsilon
}^{n}-u_{0}^{n})A_{\epsilon }\nabla (u_{\epsilon }^{n}-u_{0}^{n})\cdot
\nabla \rho dx \\
=-\epsilon ^{2}\int_{\Omega }\rho \theta _{\delta }(u_{\epsilon
}^{n}-u_{0}^{n})A_{11}\nabla _{X_{1}}(u_{\epsilon }^{n}-u_{0}^{n})\cdot
\nabla _{X_{1}}\rho dx \\
-\epsilon \int_{\Omega }\rho \theta _{\delta }(u_{\epsilon
}^{n}-u_{0}^{n})A_{12}\nabla _{X_{2}}(u_{\epsilon }^{n}-u_{0}^{n})\cdot
\nabla _{X_{1}}\rho dx
\end{multline*}

where we have used that $\rho $ is independent of $X_{2}.$

Using the ellipticity assumption for the left hand side and assumption (\ref%
{4}) for the right hand side of previous equality we deduce%
\begin{multline*}
\epsilon ^{2}\lambda \int_{\Omega }\theta _{\delta }^{\prime }(u_{\epsilon
}^{n}-u_{0}^{n})\left\vert \rho \nabla _{X_{1}}(u_{\epsilon
}^{n}-u_{0}^{n})\right\vert ^{2}dx+\lambda \int_{\Omega }\theta _{\delta
}^{\prime }(u_{\epsilon }^{n}-u_{0}^{n})\left\vert \rho \nabla
_{X_{2}}(u_{\epsilon }^{n}-u_{0}^{n})\right\vert ^{2}dx \\
\leq \epsilon ^{2}C\int_{\Omega }\rho \left\vert \theta _{\delta
}(u_{\epsilon }^{n}-u_{0}^{n})\right\vert \left\vert \nabla
_{X_{1}}(u_{\epsilon }^{n}-u_{0}^{n})\right\vert dx \\
+\epsilon C\int_{\Omega }\rho \left\vert \theta _{\delta }(u_{\epsilon
}^{n}-u_{0}^{n})\right\vert \left\vert \nabla _{X_{2}}(u_{\epsilon
}^{n}-u_{0}^{n})\right\vert dx
\end{multline*}

Where $C\geq 0$ depends on $A$ and $\rho $. Using Young's inequality $ab\leq 
\dfrac{a^{2}}{2c}+c\dfrac{b^{2}}{2}$ for the two terms in the right hand
side of the previous inequality we obtain%
\begin{multline*}
\epsilon ^{2}\frac{\lambda }{2}\int_{\Omega }\theta _{\delta }^{\prime
}(u_{\epsilon }^{n}-u_{0}^{n})\left\vert \rho \nabla _{X_{1}}(u_{\epsilon
}^{n}-u_{0}^{n})\right\vert ^{2}dx+\frac{\lambda }{2}\int_{\Omega }\theta
_{\delta }^{\prime }(u_{\epsilon }^{n}-u_{0}^{n})\left\vert \rho \nabla
_{X_{2}}(u_{\epsilon }^{n}-u_{0}^{n})\right\vert ^{2}dx \\
\leq \epsilon ^{2}C^{\prime }\int_{\omega _{1}^{\prime \prime }\times \omega
_{2}}\left\vert \theta _{\delta }(u_{\epsilon }^{n}-u_{0}^{n})\right\vert
^{2}\theta _{\delta }^{\prime }(u_{\epsilon }^{n}-u_{0}^{n})^{-1}dx
\end{multline*}

Whence%
\begin{multline*}
\epsilon ^{2}\frac{\lambda }{2}\int_{\Omega }\theta _{\delta }^{\prime
}(u_{\epsilon }^{n}-u_{0}^{n})\left\vert \rho \nabla _{X_{1}}(u_{\epsilon
}^{n}-u_{0}^{n})\right\vert ^{2}dx+\frac{\lambda }{2}\int_{\Omega }\theta
_{\delta }^{\prime }(u_{\epsilon }^{n}-u_{0}^{n})\left\vert \rho \nabla
_{X_{2}}(u_{\epsilon }^{n}-u_{0}^{n})\right\vert ^{2}dx \\
\leq \frac{4}{(p-1)^{2}}\epsilon ^{2}C^{\prime }\int_{\omega _{1}^{\prime
\prime }\times \omega _{2}}(\delta +\left\vert u_{\epsilon
}^{n}-u_{0}^{n}\right\vert )^{p}dx
\end{multline*}

where $C^{\prime \prime }$ is independent of $\epsilon $ and $n$

Now, using H\"{o}lder's inequality and the previous inequality we deduce%
\begin{multline*}
\epsilon ^{2}\frac{\lambda }{2}\left\Vert \rho \nabla _{X_{1}}(u_{\epsilon
}^{n}-u_{0}^{n})\right\Vert _{L^{p}(\Omega )}^{2}+\frac{\lambda }{2}%
\left\Vert \rho \nabla _{X_{2}}(u_{\epsilon }^{n}-u_{0}^{n})\right\Vert
_{L^{p}(\Omega )}^{2} \\
\leq \left[ 
\begin{array}{c}
\epsilon ^{2}\frac{\lambda }{2}\left( \int_{\Omega }\theta _{\delta
}^{\prime }(u_{\epsilon }^{n}-u_{0}^{n})\left\vert \rho \nabla
_{X_{1}}(u_{\epsilon }^{n}-u_{0}^{n})\right\vert ^{2}dx\right) \\ 
+\frac{\lambda }{2}\left( \int_{\Omega }\theta _{\delta }^{\prime
}(u_{\epsilon }^{n}-u_{0}^{n})\left\vert \rho \nabla _{X_{2}}(u_{\epsilon
}^{n}-u_{0}^{n})\right\vert ^{2}dx\right)%
\end{array}%
\right] \times \\
\left( \int_{\omega _{1}^{\prime \prime }\times \omega _{2}}(\delta
+\left\vert u_{\epsilon }^{n}-u_{0}^{n}\right\vert )^{p}dx\right) ^{\frac{2-p%
}{p}} \\
\leq \frac{4C^{\prime }}{(p-1)^{2}}\epsilon ^{2}\left( \int_{\omega
_{1}^{\prime \prime }\times \omega _{2}}(\delta +\left\vert u_{\epsilon
}^{n}-u_{0}^{n}\right\vert )^{p}dx\right) ^{\frac{2}{p}}
\end{multline*}

Passing to the limit as $\delta \rightarrow 0$ using the Lebesgue theorem.
Passing to the limit as $n\rightarrow \infty $ we get%
\begin{eqnarray}
&&\epsilon ^{2}\left\Vert \nabla _{X_{1}}(u_{\epsilon }-u_{0})\right\Vert
_{L^{p}(\omega _{1}^{\prime }\times \omega _{2})}^{2}+\left\Vert \nabla
_{X_{2}}(u_{\epsilon }-u_{0})\right\Vert _{L^{p}(\omega _{1}^{\prime }\times
\omega _{2})}^{2}  \label{48} \\
&\leq &C^{\prime \prime }\epsilon ^{2}\left\Vert (u_{\epsilon
}-u_{0})\right\Vert _{L^{p}(\omega _{1}^{\prime \prime }\times \omega
_{2})}^{2}  \notag
\end{eqnarray}

Using Poincar\'{e}'s inequality 
\begin{equation*}
\left\Vert (u_{\epsilon }-u_{0})\right\Vert _{L^{p}(\omega _{1}^{\prime
\prime }\times \omega _{2})}\leq C_{\omega _{2}}\left\Vert \nabla
_{X_{2}}(u_{\epsilon }-u_{0})\right\Vert _{L^{p}(\omega _{1}^{\prime \prime
}\times \omega _{2})}\text{,}
\end{equation*}

we obtain%
\begin{multline*}
\epsilon ^{2}\left\Vert \nabla _{X_{1}}(u_{\epsilon }-u_{0})\right\Vert
_{L^{p}(\omega _{1}^{\prime }\times \omega _{2})}^{2}+\left\Vert \nabla
_{X_{2}}(u_{\epsilon }-u_{0})\right\Vert _{L^{p}(\omega _{1}^{\prime }\times
\omega _{2})}^{2} \\
\leq C^{\prime \prime }\epsilon ^{2}\left\Vert \nabla _{X_{2}}(u_{\epsilon
}-u_{0})\right\Vert _{L^{p}(\omega _{1}^{\prime \prime }\times \omega
_{2})}^{2}
\end{multline*}

Let $m\in 
%TCIMACRO{\U{2115} }%
%BeginExpansion
\mathbb{N}
%EndExpansion
^{\ast }$ then there exists $\omega _{1}^{\prime }\subset \subset \omega
_{1}^{\prime \prime }\subset \subset ...\omega _{1}^{(m+1)}\subset \subset
\omega _{1}$. Iterating the above inequality $m-$time we deduce%
\begin{multline*}
\epsilon ^{2}\left\Vert \nabla _{X_{1}}(u_{\epsilon }-u_{0})\right\Vert
_{L^{p}(\omega _{1}^{\prime }\times \omega _{2})}^{2}+\left\Vert \nabla
_{X_{2}}(u_{\epsilon }-u_{0})\right\Vert _{L^{p}(\omega _{1}^{\prime }\times
\omega _{2})}^{2} \\
\leq C_{m}\epsilon ^{2m}\left\Vert \nabla _{X_{2}}(u_{\epsilon
}-u_{0})\right\Vert _{L^{p}(\omega _{1}^{(m)}\times \omega _{2})}^{2}
\end{multline*}

Now, from (\ref{48}) (with $\omega _{1}^{\prime }$ and $\omega _{1}^{\prime
\prime }$ replaced by $\omega _{1}^{(m)}$ and $\omega _{1}^{(m+1)}$
respectively) we deduce%
\begin{multline*}
\epsilon ^{2}\left\Vert \nabla _{X_{1}}(u_{\epsilon }-u_{0})\right\Vert
_{L^{p}(\omega _{1}^{\prime }\times \omega _{2})}^{2}+\left\Vert \nabla
_{X_{2}}(u_{\epsilon }-u_{0})\right\Vert _{L^{p}(\omega _{1}^{\prime }\times
\omega _{2})}^{2} \\
\leq C_{m}^{\prime }\epsilon ^{2(m+1)}\left\Vert u_{\epsilon
}-u_{0}\right\Vert _{L^{p}(\omega _{1}^{(m+1)}\times \omega _{2})}^{2}
\end{multline*}

Since $u_{\epsilon }\rightarrow u_{0}$ in $L^{p}(\Omega )$ then $\left\Vert
u_{\epsilon }-u_{0}\right\Vert _{L^{p}(\Omega )}$ is bounded and therefore
we obtain 
\begin{equation*}
\left\Vert u_{\epsilon }-u_{0}\right\Vert _{W^{p}(\omega _{1}^{\prime
}\times \omega _{2})}\leq C_{m}^{\prime \prime }\epsilon ^{m}
\end{equation*}

And the proof of the theorem is finished.
\end{proof}

Can one obtain a more better convergence rate? In fact, the anisotropic
singular perturbation problem (\ref{2}) can be seen as a problem in a
cylinder becoming unbounded. Indeed the two problems can be connected to
each other via a scaling \ $\epsilon =\frac{1}{\ell }$ (see \cite{13} for
more details). So let us consider the problem%
\begin{equation}
\left\{ 
\begin{array}{cc}
-\func{div}(\tilde{A}\nabla u_{\ell })=f\text{ \ \ \ \ \ \ \ \ \ \ } &  \\ 
u_{\ell }=0\text{ \ \ on }\partial \Omega _{\ell }\text{\ \ \ \ \ \ \ \ \ \
\ \ \ \ \ \ } & \text{ }%
\end{array}%
\right.  \label{22}
\end{equation}

where $\tilde{A}=(\tilde{a}_{ij})$ is a $N\times N$ matrix such that%
\begin{equation}
\tilde{a}_{ij}\in L^{\infty }(%
%TCIMACRO{\U{211d} }%
%BeginExpansion
\mathbb{R}
%EndExpansion
^{q}\times \omega _{2})  \label{23}
\end{equation}
\begin{equation}
\exists \lambda >0:\tilde{A}\xi \cdot \xi \geq \lambda \left\vert \xi
\right\vert ^{2}\text{ }\forall \xi \in 
%TCIMACRO{\U{211d} }%
%BeginExpansion
\mathbb{R}
%EndExpansion
^{N}\text{ for a.e }x\in 
%TCIMACRO{\U{211d} }%
%BeginExpansion
\mathbb{R}
%EndExpansion
^{q}\times \omega _{2},  \label{24}
\end{equation}

$\Omega _{\ell }=\ell \omega _{1}\times \omega _{2}$ a bounded domain where $%
\omega _{1},$ $\omega _{2}$ are two bounded Lipschitz domain with $\omega
_{1}$ convex and containing $0.$

We assume that $f\in L^{p}(\omega _{2})$ ($1<p<2$) and $\tilde{A}_{22}(x)=%
\tilde{A}_{22}(X_{2})$, $\tilde{A}_{12}(x)=\tilde{A}_{12}(X_{2})$.

We consider the limit problem 
\begin{equation}
\left\{ 
\begin{array}{cc}
-\func{div}(\tilde{A}_{22}\nabla _{X_{2}}u_{\infty })=f\text{ \ \ \ \ \ \ \
\ \ \ } &  \\ 
u_{\infty }=0\text{ \ \ on }\partial \omega _{2}\text{\ \ \ \ \ \ \ \ \ \ \
\ \ \ \ \ } & \text{ }%
\end{array}%
\right.  \label{25}
\end{equation}

Then under the above assumptions we have

\begin{theorem}
Let $u_{\ell }$, $u_{\infty }$ be the unique entropy solutions to (\ref{22})
and (\ref{25}) then for every $\alpha \in (0,1)$ there exists $C\geq 0,c>0$
independent of $\ell $ such that 
\begin{equation*}
\left\Vert \nabla (u_{\ell }-u_{\infty })\right\Vert _{W^{1,p}(\Omega
_{\alpha \ell })}\leq Ce^{-c\ell }
\end{equation*}
\end{theorem}

\begin{proof}
Let $u_{\ell }$, $u_{\infty }$ the unique entropy solutions to (\ref{22})
and (\ref{25}) respectively, and let $(u_{\ell }^{n})$ and $(u_{\infty
}^{n}) $ the approximation sequences (as in section 2). we have $u_{\ell
}^{n}\rightarrow u_{\ell }$ in $W_{0}^{1,p}(\Omega _{\ell })$ and $u_{\infty
}^{n}\rightarrow u_{\infty }$ in $W_{0}^{1,p}(\omega _{2}).$Subtracting the
associated approximated problems to (\ref{22}) and (\ref{25}) and take the
weak formulation we get 
\begin{equation}
\int_{\Omega _{\ell }}\tilde{A}\nabla (u_{\ell }^{n}-u_{\infty }^{n})\nabla
\varphi dx=0\text{, }\varphi \in \mathcal{D}(\Omega )  \label{26}
\end{equation}%
Where we have used that $\tilde{A}_{22}$, $\tilde{A}_{12}$, $u_{\infty }^{n}$
are independent of $X_{1}$. Now we will use the iteration technique
introduced in \cite{14}, let $0<\ell _{0}\leq \ell -1$, and let $\rho \in 
\mathcal{D}(%
%TCIMACRO{\U{211d} }%
%BeginExpansion
\mathbb{R}
%EndExpansion
^{q})$ a bump function such that%
\begin{equation*}
0\leq \rho \leq 1\text{, }\rho =1\text{ on }\ell _{0}\omega _{1}\text{ and }%
\rho =0\text{ on }%
%TCIMACRO{\U{211d} }%
%BeginExpansion
\mathbb{R}
%EndExpansion
^{q}\diagdown (\ell _{0}+1)\omega _{1}\text{, }\left\vert \nabla
_{X_{1}}\rho \right\vert \leq c_{0}
\end{equation*}%
where $c_{0}$ is the universal constant (see \cite{13}). Testing with $\rho
^{2}\theta _{\delta }(u_{\ell }^{n}-u_{\infty }^{n})\in H_{0}^{1}(\Omega
_{\ell })$ in (\ref{26}) we get%
\begin{multline*}
\int_{\Omega _{\ell }}\rho ^{2}\theta _{\delta }^{\prime }(u_{\ell
}^{n}-u_{\infty }^{n})\tilde{A}\nabla (u_{\ell }^{n}-u_{\infty }^{n})\cdot
\nabla (u_{\ell }^{n}-u_{\infty }^{n})dx \\
+\int_{\Omega _{\ell }}\rho \theta _{\delta }(u_{\ell }^{n}-u_{\infty }^{n})%
\tilde{A}\nabla (u_{\ell }^{n}-u_{\infty }^{n})\cdot \nabla \rho dx=0
\end{multline*}%
Using the ellipticity assumption (\ref{24}) 
\begin{multline*}
\int_{\Omega _{\ell }}\rho ^{2}\theta _{\delta }^{\prime }(u_{\ell
}^{n}-u_{\infty }^{n})\left\vert \nabla (u_{\ell }^{n}-u_{\infty
}^{n})\right\vert ^{2}dx \\
\leq 2\int_{\Omega _{\ell }}\rho \left\vert \theta _{\delta }(u_{\ell
}^{n}-u_{\infty }^{n})\right\vert \left\vert \tilde{A}\nabla (u_{\ell
}^{n}-u_{\infty }^{n})\right\vert \left\vert \nabla \rho \right\vert dx
\end{multline*}%
Notice that $\nabla \rho =0$ on $\Omega _{\ell _{0}}$, and $\Omega _{\ell
_{0}}\subset \Omega _{\ell _{0}+1}$ ( since $\omega _{1}$ is convex and
containing $0$). Then by the Cauchy-Schwaz inequality we get%
\begin{multline*}
\int_{\Omega _{\ell }}\rho ^{2}\theta _{\delta }^{\prime }(u_{\ell
}^{n}-u_{\infty }^{n})\left\vert \nabla (u_{\ell }^{n}-u_{\infty
}^{n})\right\vert ^{2}dx \\
\leq 2c_{0}C\int_{\Omega _{\ell _{0}+1}\diagdown \Omega _{\ell _{0}}}\rho
\left\vert \theta _{\delta }(u_{\ell }^{n}-u_{\infty }^{n})\right\vert
\left\vert \nabla (u_{\ell }^{n}-u_{\infty }^{n})\right\vert dx \\
\leq 2c_{0}C\left( \int_{\Omega _{\ell }}\rho ^{2}\theta _{\delta }^{\prime
}(u_{\ell }^{n}-u_{\infty }^{n})\left\vert \nabla (u_{\ell }^{n}-u_{\infty
}^{n})\right\vert ^{2}dx\right) ^{\frac{1}{2}}\times \\
\left( \int_{\Omega _{\ell _{0}+1}\diagdown \Omega _{\ell _{0}}}\left\vert
\theta _{\delta }(u_{\ell }^{n}-u_{\infty }^{n})\right\vert ^{2}\theta
_{\delta }^{\prime }(u_{\ell }^{n}-u_{\infty }^{n})^{-1}dx\right) ^{\frac{1}{%
2}}
\end{multline*}%
where we have used (\ref{23}). Whence we get ( since $\rho =1$ on $\Omega
_{\ell _{0}}$) 
\begin{eqnarray*}
\int_{\Omega _{\ell _{0}}}\theta _{\delta }^{\prime }(u_{\ell
}^{n}-u_{\infty }^{n})\left\vert \nabla (u_{\ell }^{n}-u_{\infty
}^{n})\right\vert ^{2}dx &\leq &\int_{\Omega _{\ell }}\rho ^{2}\theta
_{\delta }^{\prime }(u_{\ell }^{n}-u_{\infty }^{n})\left\vert \nabla
(u_{\ell }^{n}-u_{\infty }^{n})\right\vert ^{2}dx \\
&\leq &\left( \frac{4c_{0}C}{p-1}\right) ^{2}\int_{\Omega _{\ell
_{0}+1}\diagdown \Omega _{\ell _{0}}}(\delta +\left\vert u_{\ell
}^{n}-u_{\infty }^{n}\right\vert )^{p}dx
\end{eqnarray*}

From H\"{o}lder's inequality it holds that%
\begin{multline*}
\left\Vert \nabla (u_{\ell }^{n}-u_{\infty }^{n})\right\Vert _{L^{p}(\Omega
_{\ell _{0}})}^{2} \\
\leq \left( \int_{\Omega _{\ell _{0}}}\theta _{\delta }^{\prime }(u_{\ell
}^{n}-u_{\infty }^{n})\left\vert \nabla (u_{\ell }^{n}-u_{\infty
}^{n})\right\vert ^{2}dx\right) \left( \int_{\Omega _{\ell _{0}}}(\delta
+\left\vert u_{\ell }^{n}-u_{\infty }^{n}\right\vert )^{p}dx\right) ^{\frac{%
2-p}{p}} \\
\leq \left( \frac{4c_{0}C}{p-1}\right) ^{2}\left( \int_{\Omega _{\ell
_{0}+1}\diagdown \Omega _{\ell _{0}}}(\delta +\left\vert u_{\ell
}^{n}-u_{\infty }^{n}\right\vert )^{p}dx\right) \left( \int_{\Omega _{\ell
_{0}}}(\delta +\left\vert u_{\ell }^{n}-u_{\infty }^{n}\right\vert
)^{p}dx\right) ^{\frac{2-p}{p}}
\end{multline*}

Passing to the limit as $\delta \rightarrow 0$ (using the Lebesgue theorem)
we get%
\begin{multline*}
\left\Vert \nabla (u_{\ell }^{n}-u_{\infty }^{n})\right\Vert _{L^{p}(\Omega
_{\ell _{0}})}^{2} \\
\leq C_{1}\left( \int_{\Omega _{\ell _{0}+1}\diagdown \Omega _{\ell
_{0}}}\left\vert u_{\ell }^{n}-u_{\infty }^{n}\right\vert ^{p}dx\right)
\times \left( \int_{\Omega _{\ell _{0}}}\left\vert u_{\ell }^{n}-u_{\infty
}^{n}\right\vert ^{p}dx\right) ^{\frac{2-p}{p}}\text{,}
\end{multline*}

where we have used $0\leq \rho \leq 1$. Using Poincar\'{e}'s inequality 
\begin{equation*}
\left\Vert \nabla (u_{\ell }^{n}-u_{\infty }^{n})\right\Vert _{L^{p}(\Omega
_{\ell _{0}})}\leq C_{\omega _{2}}\left\Vert \nabla (u_{\ell }^{n}-u_{\infty
}^{n})\right\Vert _{L^{p}(\Omega _{\ell _{0}})}
\end{equation*}%
we get 
\begin{equation*}
\left\Vert \nabla (u_{\ell }^{n}-u_{\infty }^{n})\right\Vert _{L^{p}(\Omega
_{\ell _{0}})}^{p}\leq C_{2}\left\Vert u_{\ell }^{n}-u_{\infty
}^{n}\right\Vert _{L^{p}(\Omega _{\ell _{0}+1}\diagdown \Omega _{\ell
_{0}})}^{p}
\end{equation*}

Using Poincar\'{e}'s inequality 
\begin{equation*}
\left\Vert u_{\ell }^{n}-u_{\infty }^{n}\right\Vert _{L^{p}(\Omega _{\ell
_{0}+1}\diagdown \Omega _{\ell _{0}})}\leq C_{\omega _{2}}\left\Vert \nabla
(u_{\ell }^{n}-u_{\infty }^{n})\right\Vert _{L^{p}(\Omega _{\ell
_{0}+1}\diagdown \Omega _{\ell _{0}})}
\end{equation*}

we get 
\begin{equation*}
\left\Vert \nabla (u_{\ell }^{n}-u_{\infty }^{n})\right\Vert _{L^{p}(\Omega
_{\ell _{0}})}^{p}\leq C_{3}\left\Vert \nabla (u_{\ell }^{n}-u_{\infty
}^{n})\right\Vert _{L^{p}(\Omega _{\ell _{0}+1}\diagdown \Omega _{\ell
_{0}})}^{p}
\end{equation*}

Whence 
\begin{equation*}
\left\Vert \nabla (u_{\ell }^{n}-u_{\infty }^{n})\right\Vert _{L^{p}(\Omega
_{\ell _{0}})}^{p}\leq \frac{C_{3}}{C_{3}+1}\left\Vert \nabla (u_{\ell
}^{n}-u_{\infty }^{n})\right\Vert _{L^{p}(\Omega _{\ell _{0}+1})}^{p}
\end{equation*}

Let $\alpha \in (0,1)$, iterating this formula starting from $\alpha \ell $
we get 
\begin{equation*}
\left\Vert \nabla (u_{\ell }^{n}-u_{\infty }^{n})\right\Vert _{L^{p}(\Omega
_{\alpha \ell })}^{p}\leq \left( \frac{C_{3}}{C_{3}+1}\right) ^{\left[
\alpha \ell \right] }\left\Vert \nabla (u_{\ell }^{n}-u_{\infty
}^{n})\right\Vert _{L^{p}(\Omega _{\alpha \ell +\left[ (1-\alpha )\ell %
\right] })}^{p}
\end{equation*}

Whence 
\begin{equation}
\left\Vert \nabla (u_{\ell }^{n}-u_{\infty }^{n})\right\Vert _{L^{p}(\Omega
_{\alpha \ell })}\leq ce^{-c^{\prime }\ell }\left\Vert \nabla (u_{\ell
}^{n}-u_{\infty }^{n})\right\Vert _{L^{p}(\Omega _{\ell })}  \label{27}
\end{equation}

where $c,c^{\prime }>0$ are independent of $\ell $ and $n.$

Now we have to estimate the right hand side of (\ref{27}). Testing with $%
\theta (u_{\ell }^{n})$ in the approximated problem associated to (\ref{22})
one can obtain as in subsection 2.1 
\begin{equation}
\left\Vert \nabla u_{\ell }^{n}\right\Vert _{L^{p}(\Omega _{\ell })}\leq
C\ell ^{\frac{q}{2}}  \label{28}
\end{equation}

Similarly testing with $\theta (u_{\infty }^{n})$ in the approximated
problem associated to (\ref{25}). we get 
\begin{equation}
\left\Vert \nabla u_{\infty }^{n}\right\Vert _{L^{p}(\Omega _{\ell })}\leq
C^{\prime }\ell ^{\frac{q}{2}}  \label{29}
\end{equation}

Replace (\ref{29}), (\ref{28}) in (\ref{27}) and passing to the limit as $%
n\rightarrow \infty $ we obtain the desired result.
\end{proof}

\begin{corollary}
Under the above assumptions then for every $\alpha \in (0,1)$ there exists $%
C\geq 0$, $c>0$ independent of $\ \epsilon $ such that%
\begin{equation*}
\left\Vert u_{\epsilon }-u_{0}\right\Vert _{W^{1,p}(\alpha \omega _{1}\times
\omega _{2})}\leq Ce^{-\dfrac{c}{\epsilon }}
\end{equation*}

where $u_{\epsilon }$, $u_{0}$ are the entropy solutions to (\ref{2}) and (%
\ref{5}) respectively
\end{corollary}

\begin{remark}
It is very difficult to prove the rate convergence theorem for general data.
When $f(x)=f_{1}(X_{2})+f_{2}(x)$ with $f_{1}\in L^{p}(\omega _{2})$ and $%
f_{2}\in W_{2}$ we only have the estimates 
\begin{multline*}
\epsilon \left\Vert \nabla _{X_{1}}(u_{\epsilon }-u_{0})\right\Vert
_{L^{p}(\omega _{1}^{\prime }\times \omega _{2})}+\left\Vert \nabla
_{X_{2}}(u_{\epsilon }-u_{0})\right\Vert _{L^{p}(\omega _{1}^{\prime }\times
\omega _{2})} \\
+\left\Vert u_{\epsilon }-u_{0}\right\Vert _{L^{p}(\omega _{1}^{\prime
}\times \omega _{2})}\leq C\epsilon
\end{multline*}%
This follows from the linearity of the equation, Theorem 5 and the $L^{2}-$%
theory \cite{3}.
\end{remark}

\section{Some Extensions to nonlinear problems and applications}

\subsection{A semilinear monotone problem}

We consider the semilinear problem 
\begin{equation}
\left\{ 
\begin{array}{cc}
-\func{div}(A_{\epsilon }\nabla u_{\epsilon })=f+a(u_{\epsilon })\text{ \ \
\ \ \ \ \ \ \ \ } &  \\ 
u_{\epsilon }=0\text{ \ \ on }\partial \Omega \text{\ \ \ \ \ \ \ \ \ \ \ \
\ \ \ \ } & \text{ }%
\end{array}%
\right.  \label{30}
\end{equation}

Where the $a:%
%TCIMACRO{\U{211d} }%
%BeginExpansion
\mathbb{R}
%EndExpansion
\rightarrow 
%TCIMACRO{\U{211d} }%
%BeginExpansion
\mathbb{R}
%EndExpansion
$ is a continuous nonincreasing function which satisfies the growth condition%
\begin{equation}
\forall x\in 
%TCIMACRO{\U{211d} }%
%BeginExpansion
\mathbb{R}
%EndExpansion
:\left\vert a(x)\right\vert \leq K(1+\left\vert x\right\vert )\text{, }K\geq
0  \label{31}
\end{equation}

and $f\in L^{p}(\Omega )$ where $1<p<2$ $,$ $f\notin L^{2}(\Omega )$ and $A$
is given as in Subsection 1.2. Clearly the Nemytskii operator $u\rightarrow
a(u)$ maps $L^{r}(\Omega )\rightarrow L^{r}(\Omega )$ continuously for every 
$1\leq r<\infty $. The passage to the limit (formally) gives the limit
problem%
\begin{equation}
\left\{ 
\begin{array}{cc}
-\func{div}_{X_{2}}(A_{22}(X_{1},\cdot )\nabla u_{0}(X_{1},\cdot
))=f(X_{1},\cdot )+a(u_{0}(X_{1},\cdot ))\text{\ \ \ \ } &  \\ 
u_{0}(X_{1},\cdot )=0\text{ \ \ on }\partial \Omega _{X_{1}}\text{\ \ \ \ \
\ \ \ \ \ \ \ \ \ \ \ \ \ \ \ \ \ \ \ \ \ \ \ \ \ \ \ \ \ \ \ \ \ \ \ \ } & 
\text{ }%
\end{array}%
\right.  \label{32}
\end{equation}

We can suppose that $a(0)=0$. Indeed, in the general case the right hand
side of (\ref{30}) can be replaced by $(a(0)+f)+b(x)$ where $b(x)=a(x)-a(0)$%
. Clearly $b$ is continuous nonincreasing and satisfies $\left\vert
b(x)\right\vert \leq (K+\left\vert a(0)\right\vert )(1+\left\vert
x\right\vert )$.

First of all, suppose that $f\in L^{2}(\Omega ),$then we have the following

\begin{proposition}
Assume (\ref{3}), (\ref{4}) and $a(0)=0$. Let $u_{\epsilon }$ be the unique
weak solution in $H_{0}^{1}(\Omega )$ to (\ref{30}) then $\epsilon \nabla
_{X_{1}}u_{\epsilon }\rightarrow 0$ in $L^{2}(\Omega )$ and $u_{\epsilon
}\rightarrow u_{0}$ in $V_{2}$ where $u_{0}$ in the unique solution in $%
V_{2} $ to the limit problem (\ref{32}).
\end{proposition}

\begin{proof}
Existence of $u_{\epsilon }$ follows directly by a simple application of the
Schauder fixed point theorem for example. The uniqueness follows form
monotonicity of $a$ and the Poincar\'{e}'s inequality.

Take $u_{\epsilon }$ as a test function in (\ref{30}) then one can obtain
the estimates 
\begin{equation*}
\epsilon \left\Vert \nabla _{X_{1}}u_{\epsilon }\right\Vert _{L^{2}(\Omega )}%
\text{, }\left\Vert \nabla _{X_{2}}u_{\epsilon }\right\Vert _{L^{2}(\Omega )}%
\text{, }\left\Vert u_{\epsilon }\right\Vert _{L^{2}(\Omega )}\leq C\text{,}
\end{equation*}

where $C$ is independent of $\epsilon $, we have used that $\int_{\Omega
}a(u_{\epsilon })u_{\epsilon }dx\leq 0$ (thanks to monotonicity assumption
and $a(0)=0$). And we also have (thanks to assumption (\ref{31}))%
\begin{equation*}
\left\Vert a(u_{\epsilon })\right\Vert _{L^{2}(\Omega )}\leq K(\left\vert
\Omega \right\vert ^{\frac{1}{2}}+C)
\end{equation*}

so there exists $v\in L^{2}(\Omega )$, $u_{0}\in L^{2}(\Omega )$, $\nabla
_{X_{2}}u_{0}\in L^{2}(\Omega )$ and a subsequence $(u_{\epsilon
_{k}})_{k\in 
%TCIMACRO{\U{2115} }%
%BeginExpansion
\mathbb{N}
%EndExpansion
}$ such that 
\begin{equation}
a(u_{\epsilon _{k}})\rightarrow v\text{, }\epsilon _{k}\nabla
_{X_{1}}u_{\epsilon _{k}}\rightharpoonup 0\text{, }\nabla
_{X_{2}}u_{\epsilon _{k}}\rightharpoonup \nabla _{X_{2}}u_{0}\text{, }%
u_{\epsilon _{k}}\rightharpoonup u_{0}\text{ in }L^{2}(\Omega )\text{-weak}
\label{33}
\end{equation}

Passing to the in the weak formulation of (\ref{30}) we get%
\begin{equation}
\int_{\Omega }A_{22}\nabla _{X_{2}}u_{0}\cdot \nabla _{X_{2}}\varphi
dx=\int_{\Omega }f\varphi dx\ +\int_{\Omega }v\varphi dx\text{,}\ \varphi
\in \mathcal{D(}\Omega \mathcal{)}  \label{34}
\end{equation}

Take $\varphi =u_{\epsilon _{k}}$ in the previous equality and passing to
the limit we get%
\begin{equation}
\int_{\Omega }A_{22}\nabla _{X_{2}}u_{0}\cdot \nabla
_{X_{2}}u_{0}dx=\int_{\Omega }fu_{0}dx\ +\int_{\Omega }vu_{0}dx  \label{35}
\end{equation}

Let us computing the quantity 
\begin{multline*}
0\leq I_{k}=\int_{\Omega }A_{\epsilon _{k}}\left( 
\begin{array}{c}
\nabla _{X_{1}}u_{\epsilon _{k}} \\ 
\nabla _{X_{2}}(u_{\epsilon _{k}}-u_{0})%
\end{array}%
\right) \cdot \left( 
\begin{array}{c}
\nabla _{X_{1}}u_{\epsilon _{k}} \\ 
\nabla _{X_{2}}(u_{\epsilon _{k}}-u_{0})%
\end{array}%
\right) dx \\
-\int_{\Omega }(a(u_{\epsilon _{k}})-a(u_{0}))(u_{\epsilon _{k}}-u_{0})dx \\
=\int_{\Omega }fu_{\epsilon _{k}}dx-\epsilon \int_{\Omega }A_{12}\nabla
_{X_{2}}u_{0}\cdot \nabla _{X_{1}}u_{\epsilon _{k}}dx-\epsilon \int_{\Omega
}A_{21}\nabla _{X_{1}}u_{\epsilon _{k}}\cdot \nabla _{X_{2}}u_{0}dx \\
-\int_{\Omega }A_{22}\nabla _{X_{2}}u_{\epsilon _{k}}\cdot \nabla
_{X_{2}}u_{0}dx-\int_{\Omega }A_{22}\nabla _{X_{2}}u_{0}\cdot \nabla
_{X_{2}}u_{\epsilon _{k}}dx \\
+\int_{\Omega }fu_{0}dx+\int_{\Omega }vu_{0}dx+\int_{\Omega
}a(u_{0})u_{\epsilon _{k}}dx \\
+\int_{\Omega }a(u_{\epsilon _{k}})u_{0}dx-\int_{\Omega }a(u_{0})u_{0}dx
\end{multline*}

(This quantity is positive thanks to the ellipticity and monotonicity
assumptions).

Passing to the limit as $k\rightarrow \infty $ using (\ref{33}), (\ref{34}),
(\ref{35}) we get 
\begin{equation*}
\lim I_{k}=0
\end{equation*}

And finally The ellipticity assumption and Poincar\'{e}'s inequality show
that 
\begin{equation}
\left\Vert \epsilon _{k}\nabla _{X_{1}}u_{\epsilon _{k}}\right\Vert
_{L^{2}(\Omega )}\text{, }\left\Vert \nabla _{X_{2}}(u_{\epsilon
_{k}}-u_{0})\right\Vert _{L^{2}(\Omega )}\text{, }\left\Vert u_{\epsilon
_{k}}-u_{0}\right\Vert _{L^{2}(\Omega )}\rightarrow 0  \label{36}
\end{equation}

Whence (\ref{34}) becomes%
\begin{equation}
\int_{\Omega }A_{22}\nabla _{X_{2}}u_{0}\cdot \nabla _{X_{2}}\varphi
dx=\int_{\Omega }f\varphi dx\ +\int_{\Omega }a(u_{0})\varphi dx\text{,}\
\varphi \in \mathcal{D(}\Omega \mathcal{)}  \label{37}
\end{equation}

$\left\Vert \nabla _{X_{2}}(u_{\epsilon _{k}}-u_{0})\right\Vert
_{L^{2}(\Omega )}\rightarrow 0$ shows that $u_{0}\in V_{2}$, and therefore%
\begin{equation*}
\int_{\Omega _{X_{1}}}A_{22}\nabla _{X_{2}}u_{0}\cdot \nabla _{X_{2}}\varphi
dx=\int_{\Omega _{X_{1}}}f\varphi dx\ +\int_{\Omega _{X_{1}}}a(u_{0})\varphi
dx\text{,}\ \varphi \in \mathcal{D(}\Omega _{X_{1}}\mathcal{)}
\end{equation*}

Hence $u_{0}(X_{1},\cdot )$ is a solution to (\ref{32}). The uniqueness in $%
H_{0}^{1}(\Omega _{X_{1}})$ of the the solution of the limit problem (\ref%
{32}) shows that $u_{0}$ is the unique function in $V_{2}$ which satisfies (%
\ref{37}). Therefore the convergences (\ref{36}) hold for the whole sequence 
$(u_{\epsilon })_{0<\epsilon \leq 1}.$
\end{proof}

Now, we are ready to give the main result of this subsection

\begin{theorem}
Suppose that $f\in L^{p}(\Omega )$ where $1<p<2$ (we can suppose that $%
f\notin L^{2}(\Omega )$) then there exists $u_{0}\in V_{p}$ such that $%
u_{0}(X_{1},\cdot )$ is the unique entropy solution to (\ref{32}) and we
have $u_{\epsilon }\rightarrow u_{0}$ in $V_{p}$, $\epsilon \nabla
_{X_{1}}u_{\epsilon }\rightarrow 0$ in $L^{p}(\Omega )$, where $u_{\epsilon
} $ is the unique entropy solution to (\ref{30}).
\end{theorem}

\begin{proof}
We only give a sketch of the proof. Existence and uniqueness of the entropy
solutions to (\ref{30}) and (\ref{32}) follows from the general result
proved in \cite{2}. As in proof of Theorem 2 we shall construct the entropy
solution $u_{\epsilon }$. we consider the approximated problem%
\begin{equation*}
\left\{ 
\begin{array}{cc}
-\func{div}(A_{\epsilon }\nabla u_{\epsilon }^{n})=f_{n}+a(u_{\epsilon }^{n})%
\text{ \ \ \ \ \ \ \ \ \ \ } &  \\ 
u_{\epsilon }^{n}=0\text{ \ \ on }\partial \Omega \text{\ \ \ \ \ \ \ \ \ \
\ \ \ \ \ \ \ \ \ \ \ \ \ } & \text{ }%
\end{array}%
\right.
\end{equation*}

We follows the same arguments as in section 2, where we use the above
proposition and the following 
\begin{equation*}
\int_{\Omega }(a(u)-a(v)\theta (u-v)dx\leq 0
\end{equation*}

Which holds for every $u,v\in L^{2}(\Omega )$, in fact this follows from
monotonicity of $a$ and $\theta $.
\end{proof}

\subsection{Nonlinear problem without monotonicity assumption}

Suppose that $\Omega =\omega _{1}\times \omega _{2}$ where $\omega _{1}$, $%
\omega _{2}$ and consider the following nonlinear problem%
\begin{equation}
\left\{ 
\begin{array}{cc}
-\func{div}(A_{\epsilon }\nabla u_{\epsilon })=f+B(u_{\epsilon })\text{\ \ \
\ \ \ \ \ \ \ \ \ \ \ \ \ \ \ \ \ \ \ \ \ \ \ \ } &  \\ 
u_{\epsilon }=0\text{ \ \ on }\partial \Omega \text{ \ \ \ \ \ \ \ \ \ \ \ \
\ \ \ \ \ \ \ \ \ \ \ \ \ \ \ \ \ \ \ \ \ \ \ \ \ \ \ \ \ \ } & \text{ }%
\end{array}%
\right.  \label{38}
\end{equation}%
\bigskip Where $f\in L^{p}(\Omega )$, $1<p<2$ and $B:L^{p}(\Omega
)\rightarrow L^{p}(\Omega )$ is a continuous nonlinear operator. We suppose
that%
\begin{equation}
\exists M\geq 0\text{, }\forall u\in L^{p}(\Omega ):\left\Vert
B(u)\right\Vert _{L^{p}}\leq M  \label{39}
\end{equation}

\begin{proposition}
Assume (\ref{3}), (\ref{4}),and (\ref{39}) then:

1) There exists a sequence $(u_{\epsilon })_{0<\epsilon \leq 1}\subset
W_{0}^{1,p}(\Omega )$ of an entropy solutions to (\ref{38}) which are also a
weak solutions such that 
\begin{equation*}
\epsilon \left\Vert \nabla _{X_{1}}u_{\epsilon }\right\Vert _{L^{p}(\Omega )}%
\text{, }\left\Vert \nabla _{X_{2}}u_{\epsilon }\right\Vert _{L^{p}(\Omega )}%
\text{, }\left\Vert u_{\epsilon }\right\Vert _{L^{p}(\Omega )}\leq C_{0}%
\text{,}
\end{equation*}

where $C_{0}\geq 0$ is independent of $\epsilon $( the constant $C_{0}$
depends only on $\Omega $, $\lambda $, $f$ and $M$).

2) If $(u_{\epsilon })_{0<\epsilon \leq 1}$ is a sequence of entropy and
weak solutions to (\ref{38}) then we have the above estimates.
\end{proposition}

\begin{proof}
1) The existence of $u_{\epsilon }$ is based on the Schauder fixed point
theorem, we define the mapping $\Gamma :L^{p}(\Omega )\rightarrow
L^{p}(\Omega )$ by 
\begin{equation*}
v\in L^{p}(\Omega )\rightarrow \Gamma (v)=v_{\epsilon }\in
W_{0}^{1,p}(\Omega )
\end{equation*}

where $v_{\epsilon }$ is the entropy solution of the linearized problem%
\begin{equation}
\left\{ 
\begin{array}{cc}
-\func{div}(A_{\epsilon }\nabla v_{\epsilon })=f+B(v)\text{\ \ \ \ \ \ \ \ \
\ \ \ \ \ \ \ \ \ \ \ \ \ \ \ \ \ \ \ \ } &  \\ 
v_{\epsilon }=0\text{ \ \ on }\partial \Omega \text{ \ \ \ \ \ \ \ \ \ \ \ \
\ \ \ \ \ \ \ \ \ \ \ \ \ \ \ \ \ \ \ \ \ \ \ \ \ \ \ \ \ \ } & \text{ }%
\end{array}%
\right.  \label{40}
\end{equation}

Since the entropy solution is unique then $\Gamma $ is well defined. we can
prove easily (by using the approximation method) that $\Gamma $ is
continuous. As in subsection 2.1 we can obtain the estimates%
\begin{equation*}
\epsilon \left\Vert \nabla _{X_{1}}u_{\epsilon }\right\Vert _{L^{p}(\Omega )}%
\text{, }\left\Vert \nabla _{X_{2}}u_{\epsilon }\right\Vert _{L^{p}(\Omega )}%
\text{, }\left\Vert u_{\epsilon }\right\Vert _{L^{p}(\Omega )}\leq C_{0}
\end{equation*}

where $C_{0}$ is independent of $\epsilon $ and $v$ (thanks to (\ref{39}))

Now, define the subset%
\begin{equation*}
K=\left\{ u\in W_{0}^{1,p}(\Omega ):\epsilon \left\Vert \nabla
_{X_{1}}u\right\Vert _{L^{p}(\Omega )}\text{, }\left\Vert \nabla
_{X_{2}}u\right\Vert _{L^{p}(\Omega )}\text{, }\left\Vert u\right\Vert
_{L^{p}(\Omega )}\leq C_{0}\right\}
\end{equation*}

The subset $K$ is convex and compact in $L^{p}(\Omega )$ thanks to the
Sobolev compact embedding $W_{0}^{1,p}(\Omega )\subset L^{p}(\Omega ).$

The subset $K$ is stable under $\Gamma $ (since $C_{0}$ is independent of $v$
as mentioned above). Whence $\Gamma $ admits at least a fixed point $%
u_{\epsilon }\in K,$ in other words $u_{\epsilon }$ is a weak solution to (%
\ref{38}) which is also an entropy solution, this last assertion follows
from the definition of $\Gamma $.

2) Let $(u_{\epsilon })_{0<\epsilon \leq 1}$ be a sequence of entropy and
weak solutions to (\ref{38}) \ $u_{\epsilon }$ is the unique entropy
solution to (\ref{40}) with $v$ replaced by $u_{\epsilon }$ and therefore we
obtain the desired estimates as proved above.
\end{proof}

\begin{remark}
In the general case the entropy solution $u_{\epsilon }$ of (\ref{38}) is
not necessarily unique.
\end{remark}

Now, assume that 
\begin{equation}
f(x)=f(X_{2})\text{, }A_{22}(x)=A_{22}(X_{2})\text{, }A_{12}(x)=A_{12}(X_{2})
\label{41}
\end{equation}

And assume that for every $E\subset W_{p}$ bounded in $L^{p}(\Omega )$ we
have%
\begin{equation}
\overline{conv}\left\{ B(E)\right\} \subset W_{2}\text{,}  \label{42}
\end{equation}

where $\overline{conv}\left\{ B(E)\right\} $ is the closed convex-hull of $%
B(E)$ in $L^{p}(\Omega )$. Assumption (\ref{42}) appears strange. We shall
give later some concrete examples of operators which satisfy this
assumption. Let us prove the following

\begin{theorem}
Assume (\ref{3}), (\ref{4}), (\ref{39}), (\ref{41}) and (\ref{42}). Let $%
(u_{\epsilon })_{0<\epsilon \leq 1}\subset W_{0}^{1,p}(\Omega )$ be an
entropy and weak solution to (\ref{38}) then for every $\Omega ^{\prime
}\subset \subset \Omega $ there exists $C_{\Omega ^{\prime }}\geq 0$
independent of $\epsilon $ such that 
\begin{equation*}
\forall \epsilon :\left\Vert u_{\epsilon }\right\Vert _{W^{1,p}(\Omega
^{\prime })}\leq C_{\Omega ^{\prime }}
\end{equation*}
\end{theorem}

\begin{proof}
The proof is similar the one given in our preprint \cite{10}. Let $(\Omega
_{i})_{j\in 
%TCIMACRO{\U{2115} }%
%BeginExpansion
\mathbb{N}
%EndExpansion
}$ an open covering of $\Omega $ such that $\overline{\Omega _{j}}\subset
\Omega _{j+1}$. We equip the space $Z=W_{loc}^{1,p}(\Omega )$ with the
topology generated by the family of seminorms $(p_{j})_{j\in 
%TCIMACRO{\U{2115} }%
%BeginExpansion
\mathbb{N}
%EndExpansion
}$ defined by 
\begin{equation*}
p_{j}(u)=\left\Vert u_{\epsilon }\right\Vert _{W^{1,p}(\Omega _{j})}
\end{equation*}

Equipped with this topology, $Z$ is a separated locally convex topological
vector space. We set $Y=L^{p}(\Omega )$ equipped with its natural topology.
We define the family of the linear continuous mappings 
\begin{equation*}
\Lambda _{\epsilon }:Y\rightarrow Z
\end{equation*}

defined by: $g\in Y$, $\Lambda _{\epsilon }(g)=v_{\epsilon }$ where $%
v_{\epsilon }$ is the unique entropy solution to 
\begin{equation*}
\left\{ 
\begin{array}{cc}
-\func{div}(A_{\epsilon }\nabla v_{\epsilon })=g\text{\ \ \ \ \ \ \ \ \ \ \
\ \ \ \ \ \ \ \ \ \ \ \ \ \ \ \ \ } &  \\ 
v_{\epsilon }=0\text{ \ \ on }\partial \Omega \text{ \ \ \ \ \ \ \ \ \ \ \ \
\ \ \ \ \ \ \ \ \ \ \ \ \ \ \ \ \ } & \text{ }%
\end{array}%
\right.
\end{equation*}

The continuity of $\Lambda _{\epsilon }$ follows immediately if we observe $%
\Lambda _{\epsilon }$ as a composition of $\Lambda _{\epsilon }:Y\rightarrow
Y$ and the canonical injection $Y$ $\rightarrow Z$

Now, we denote $Z_{w}$, $Y_{w}$ the spaces $Z$, $Y$ equipped with the weak
topology respectively. then $\Lambda _{\epsilon }:Y_{w}\rightarrow Z_{w}$ is
also continuous.

Consider the bounded (in $Y$) subset%
\begin{equation*}
E_{0}=\left\{ u\in W_{p}\mid \left\Vert u\right\Vert _{L^{p}(\Omega )}\leq
C_{0}\right\} ,
\end{equation*}

where $C_{0}$ is the constant introduced in Proposition 5. Consider the
subset $G=f+\overline{conv}\left\{ B(E_{0})\right\} $ where the closure is
taken in the $L^{p}-$topology. Thanks to assumption (\ref{42}) and (\ref{39}%
) $G$ is closed convex and bounded in $Y$. Now for every $g\in G$ the orbit $%
\left\{ \Lambda _{\epsilon }g\right\} _{\epsilon }$ is bounded in $Z$ thanks
to Remark 2. And therefore $\left\{ \Lambda _{\epsilon }g\right\} _{\epsilon
}$ is bounded in $Z_{w}$.

Clearly the set $G$ is compact in $Y_{w}$. Then it follows by the
Banach-Steinhaus theorem (applied on the quadruple $\Lambda _{\epsilon }$, $%
G $, $Y_{w}$, $Z_{w}$) that there exists a bounded subset $F$ in $Z_{w}$
such that%
\begin{equation*}
\forall \epsilon :\text{ }\Lambda _{\epsilon }(G)\subset F
\end{equation*}

The boundedness of $F$ in $Z_{w}$ implies its boundedness in $Z$.i.e For
every $j\in 
%TCIMACRO{\U{2115} }%
%BeginExpansion
\mathbb{N}
%EndExpansion
$ there exists $C_{j}\geq 0$ independent of $\epsilon $ such that%
\begin{equation*}
\forall \epsilon :p_{j}(\Lambda _{\epsilon }(G))\leq C_{j}
\end{equation*}

Let $u_{\epsilon }$ be an entropy and weak solution to (\ref{38}) then we
have $u_{\epsilon }\in E_{0}$ as proved in Proposition 5 then $\Lambda
_{\epsilon }(f+B(u_{\epsilon }))=u_{\epsilon }$ $\in F$ for every $\epsilon $%
, therefore%
\begin{equation*}
\forall \epsilon :\left\Vert u_{\epsilon }\right\Vert _{W^{1,p}(\Omega
_{j})}\leq C_{j}
\end{equation*}

Whence for every $\Omega ^{\prime }\subset \subset \Omega $ there exists $%
C_{\Omega ^{\prime }}\geq 0$ independent of $\epsilon $ such that 
\begin{equation*}
\forall \epsilon :\left\Vert u_{\epsilon }\right\Vert _{W^{1,p}(\Omega
^{\prime })}\leq C_{\Omega ^{\prime }}
\end{equation*}
\end{proof}

Now we are ready to prove the convergence theorem. Assume that%
\begin{equation}
B:(L^{p}(\Omega ),\tau _{L_{loc}^{p}})\rightarrow L^{p}(\Omega )\text{ is
continuous}  \label{43}
\end{equation}

where $(L^{p}(\Omega ),\tau _{L_{loc}^{p}})$ is the space $L^{p}(\Omega )$
equipped with the $L_{loc}^{p}(\Omega )$-topology. Notice that (\ref{43})
implies that $B:L^{p}(\Omega )\rightarrow L^{p}(\Omega )$ is continuous.
Then we have the following

\begin{theorem}
Under assumptions of Theorem 8, assume in addition (\ref{43}), suppose that $%
\Omega $ is convex, then there exists $u_{0}\in V_{p}$ and a sequence $%
(u_{\epsilon _{k}})_{k\in 
%TCIMACRO{\U{2115} }%
%BeginExpansion
\mathbb{N}
%EndExpansion
}$ of entropy and weak solution to (\ref{38}) such that 
\begin{eqnarray*}
\epsilon _{k}\nabla _{X_{1}}u_{\epsilon _{k}} &\rightharpoonup &0\text{, }%
\nabla _{X_{2}}u_{\epsilon _{k}}\rightharpoonup \nabla _{X_{2}}u_{0}\text{
in }L^{p}(\Omega )-weak \\
\text{ and \ }u_{\epsilon _{k}} &\rightarrow &u_{0}\text{ in }%
L_{loc}^{p}(\Omega )-strong
\end{eqnarray*}

Moreover $u_{0}$ satisfies in $\mathcal{D}^{\prime }(\omega _{2})$ the
equation 
\begin{equation*}
-\func{div}_{X_{2}}(A_{22}\nabla _{X_{2}}u_{0}(X_{1},\cdot
))=f+B(u_{0})(X_{1},\cdot )
\end{equation*}

for a.e $X_{1}\in \omega _{1}$
\end{theorem}

\begin{proof}
The estimates given in Proposition 5 show that there exists $u_{0}\in
L^{p}(\Omega )$ and a sequence $(u_{\epsilon _{k}})_{k\in 
%TCIMACRO{\U{2115} }%
%BeginExpansion
\mathbb{N}
%EndExpansion
}$ solutions to (\ref{38}) such that%
\begin{equation}
\epsilon _{k}\nabla _{X_{1}}u_{\epsilon _{k}}\rightharpoonup 0\text{, }%
\nabla _{X_{2}}u_{\epsilon _{k}}\rightharpoonup \nabla _{X_{2}}u_{0}\text{
and }u_{\epsilon _{k}}\rightharpoonup u_{0}\text{ in }L^{p}(\Omega )-weak
\label{44}
\end{equation}

As we have proved in Theorem 3 we have $u_{0}\in V_{p}$. The particular
difficulty is the passage to the limit in the nonlinear term. This assertion
is guaranteed by Theorem 8. Indeed, since $\Omega $ is convex and Lipschitz
then there an open covering $(\Omega _{j})_{j\in 
%TCIMACRO{\U{2115} }%
%BeginExpansion
\mathbb{N}
%EndExpansion
}$, $\Omega _{j}\subset \Omega _{j+1}$ and $\overline{\Omega _{j}}\subset
\Omega $ such that each $\Omega _{j}$ is a Lipschitz domain (Take an
increasing sequence of number $0<\beta _{j}<1$ with $\lim \beta _{j}=1$.\
Fix $x_{0}\in \Omega $ and take $\Omega _{j}=\beta _{j}(\Omega -x_{0})+x_{0}$%
, since $\Omega $ is convex then $\overline{\Omega _{j}}\subset \Omega $.
The Lipschitz character is conserved since the multiplication by $\beta _{j}$
and translations are $C^{\infty }$ diffeomorphisms).

Theorem 8 shows that for every $j\in 
%TCIMACRO{\U{2115} }%
%BeginExpansion
\mathbb{N}
%EndExpansion
$ there exists $C_{j}\geq 0$ such that%
\begin{equation*}
\left\Vert u_{\epsilon }\right\Vert _{W^{1,p}(\Omega _{j})}\leq C_{\Omega
_{j}}
\end{equation*}

Since $\Omega _{j}$ is Lipschitz then the embedding $W^{1,p}(\Omega
_{j})\hookrightarrow L^{p}(\Omega _{j})$ is compact \cite{15} and therefore
for each $k$ there exists a subsequence $(u_{\epsilon _{k}^{j}})_{k}\subset
L^{p}(\Omega _{j})$ such that 
\begin{equation*}
u_{\epsilon _{k}^{j}}\mid _{\Omega _{j}}\rightarrow u_{0}\mid _{\Omega _{j}}
\end{equation*}

By the diagonal process one can construct a sequence $(u_{\epsilon
_{k}})_{k} $ such that $u_{\epsilon _{k}}\rightarrow u_{0}$ in $L^{p}(\Omega
_{j})$ for every $j$, in other words we have 
\begin{equation}
u_{\epsilon _{k}}\rightarrow u_{0}\text{ in }L_{loc}^{p}(\Omega )-strong
\label{45}
\end{equation}%
Now passing to the limit in the weak formulation of (\ref{38}) we deduce 
\begin{equation*}
-\func{div}_{X_{2}}(A_{22}\nabla _{X_{2}}u_{0}(X_{1},\cdot
))=f+B(u_{0})(X_{1},\cdot )\text{,}
\end{equation*}

where we have used (\ref{44}) for the passage to the limit in the left hand
side. For the passage to the limit in the nonlinear term we have used (\ref%
{45}) and assumption (\ref{43}).
\end{proof}

\begin{example}
\bigskip We give a concrete example of application of the above abstract
analysis. Let $\Omega =\omega _{1}\times \omega _{2}$ be a Lispchitz convex
domain of $%
%TCIMACRO{\U{211d} }%
%BeginExpansion
\mathbb{R}
%EndExpansion
^{q}\times 
%TCIMACRO{\U{211d} }%
%BeginExpansion
\mathbb{R}
%EndExpansion
^{N-q}$ and let $A$ be a bounded $(N-q)\times (N-q)$ matrix defined on $%
\omega _{2}$ which satisfies the ellipticity assumption. Let us consider the
integro-differential problem 
\begin{equation}
\left\{ 
\begin{array}{cc}
-\func{div}_{X_{2}}(A(X_{2})\nabla _{X_{2}}u)=f(X_{2})+\dint_{\omega
_{1}}h(X_{1}^{\prime },X_{1},X_{2})a(u(X_{1}^{\prime },X_{2}))dX_{1}^{\prime
}\text{\ \ \ \ \ \ \ \ \ \ \ \ \ \ \ \ \ \ \ \ \ \ \ \ \ \ } &  \\ 
u(X_{1},\cdot )=0\text{ \ \ on }\partial \omega _{2}\text{ \ \ \ \ \ \ \ \ \
\ \ \ \ \ \ \ \ \ \ \ \ \ \ \ \ \ \ \ \ \ \ \ \ \ \ \ \ \ \ \ \ \ \ \ \ \ \
\ \ \ \ \ \ \ \ \ \ \ \ \ \ \ \ \ \ \ \ \ \ \ \ \ \ \ \ \ \ \ \ \ \ \ } & 
\text{ }%
\end{array}%
\right.  \label{46}
\end{equation}
\end{example}

where $h\in L^{\infty }(\omega _{1}\times \Omega )$ and $f\in L^{p}(\omega
_{2})$, $1<p<2$, and $a$ is a continuous real bounded function.

This equation is based on the Neutron transport equation (see for instance 
\cite{11})

A solution to (\ref{46}) is a function $u\in V_{p}$ Which satisfies (\ref{46}%
) in $\mathcal{D}^{\prime }(\omega _{2})$. suppose that 
\begin{equation*}
\nabla _{X_{1}}h(X_{1}^{\prime },X_{1},X_{2})\in L^{\infty }(\omega
_{1}\times \Omega )
\end{equation*}

Then we have

\begin{theorem}
Under the assumptions of this example, (\ref{46}) has at least a solution in 
$V_{p}$ in the sense of $\mathcal{D}^{\prime }(\omega _{2})$ for a.e $%
X_{1}\in \omega _{1}$
\end{theorem}

\begin{proof}
We introduce the singular perturbation problem 
\begin{equation*}
\left\{ 
\begin{array}{cc}
-\func{div}_{X}(A_{\epsilon }\nabla u_{\epsilon })=f(X_{2})+\dint_{\omega
_{1}}h(X_{1}^{\prime },X_{1},X_{2})a(u_{\epsilon }(X_{1}^{\prime
},X_{2}))dX_{1}^{\prime }\text{\ \ \ \ \ \ \ \ \ \ \ \ \ \ \ \ \ \ \ \ \ \ \
\ \ \ } &  \\ 
u_{\epsilon }=0\text{ \ \ on }\partial \Omega \text{ \ \ \ \ \ \ \ \ \ \ \ \
\ \ \ \ \ \ \ \ \ \ \ \ \ \ \ \ \ \ \ \ \ \ \ \ \ \ \ \ \ \ \ \ \ \ \ \ \ \
\ \ \ \ \ \ \ \ \ \ \ \ \ \ \ \ \ \ \ \ \ \ \ \ \ \ \ \ \ \ \ \ \ \ \ \ \ }
& \text{ }%
\end{array}%
\right.
\end{equation*}

where 
\begin{equation*}
A_{\epsilon }=\left( 
\begin{array}{cc}
\epsilon ^{2}I & 0 \\ 
0 & A%
\end{array}%
\right)
\end{equation*}%
Clearly $A_{\epsilon }$ satisfies the ellipticity assumption and it is Clear
that the operator 
\begin{equation*}
u\rightarrow \dint_{\omega _{1}}h(X_{1}^{\prime
},X_{1},X_{2})a(u(X_{1}^{\prime },X_{2}))dX_{1}^{\prime }
\end{equation*}

satisfies assumption (\ref{39}).

We can prove easily that the above operator satisfies assumption (\ref{43}).
Indeed, let $u_{n}\rightarrow u$ in $L_{loc}^{p}(\Omega )$ then there exists
a subsequence $(u_{n_{k}})$ (constructed by the diagonal process) such that $%
u_{n_{k}}\rightarrow u$ a.e in $\Omega $. Since $a$ is bounded then it
follows by the Lebesgue theorem that 
\begin{equation*}
\dint_{\omega _{1}}h(X_{1}^{\prime },X_{1},X_{2})a(u_{n_{k}}(X_{1}^{\prime
},X_{2}))dX_{1}^{\prime }\rightarrow \dint_{\omega _{1}}h(X_{1}^{\prime
},X_{1},X_{2})a(u(X_{1}^{\prime },X_{2}))dX_{1}^{\prime },
\end{equation*}

in $L^{p}(\Omega )$. Whence by a contradiction argument we get 
\begin{equation*}
\dint_{\omega _{1}}h(X_{1}^{\prime },X_{1},X_{2})a(u_{n}(X_{1}^{\prime
},X_{2}))dX_{1}^{\prime }\rightarrow \dint_{\omega _{1}}h(X_{1}^{\prime
},X_{1},X_{2})a(u(X_{1}^{\prime },X_{2}))dX_{1}^{\prime },
\end{equation*}

in $L^{p}(\Omega )$

We can prove similarly as in \cite{10} that (\ref{42}) holds, therefore the
assertion of the theorem is a simple application of theorem 9
\end{proof}

\begin{remark}
Notice that the compacity of the operator given in the previous example is
not sufficient to prove a such result as in the $L^{2}$ theory \cite{11}.
This shows the importance of assumption (\ref{42}) wich holds for the above
operator.
\end{remark}

Does operator whose assumption (\ref{42}) holds admit necessarily an
integral representation as in (\ref{46})?.

\begin{example}
We shall replace the integral by a general linear operator. Let us consider
the following problem: Find $u\in V_{p}$ such that 
\begin{equation}
\left\{ 
\begin{array}{cc}
-\func{div}_{X_{2}}(A\nabla _{X_{2}}u)=f(X_{2})+gP(ha(u))\text{\ \ \ \ \ \ \
\ \ \ \ \ \ \ \ \ \ \ \ \ \ \ \ \ \ \ \ \ \ \ \ \ \ \ \ \ \ \ \ \ \ \ \ \ \
\ \ \ \ \ \ \ \ \ \ \ \ \ \ \ \ \ \ \ \ \ \ } &  \\ 
u(X_{1},\cdot )=0\text{ \ \ on }\partial \omega _{2}\text{ \ \ \ \ \ \ \ \ \
\ \ \ \ \ \ \ \ \ \ \ \ \ \ \ \ \ \ \ \ \ \ \ \ \ \ \ \ \ \ \ \ \ \ \ \ \ \
\ \ \ \ \ \ \ \ \ \ \ \ \ \ \ \ \ \ \ \ \ \ \ \ \ \ \ \ \ \ \ \ \ \ \ } & 
\text{ }%
\end{array}%
\right. ,  \label{47}
\end{equation}
\end{example}

where $a$, $A$ and $f$ are defined as in Example 1.

We suppose that $g$, $h\in L^{\infty }(\Omega )$ with $Supp(h)\subset \Omega 
$ \ compact. Assume $\nabla _{X_{1}}g\in L^{\infty }(\Omega )$ and $%
P:L^{p}(\Omega )\rightarrow L^{2}(\omega _{2})$ is a bounded linear operator.

When $P$ is not compact then the operator $u\rightarrow gP(ha(u))$ is not
necessarily compact, if this is the case then this operator cannot admit an
integral representation.

\begin{theorem}
Under the assumptions of this example there exists at least a solution $u\in
V_{p}$ to (\ref{47}) in the sense of $\mathcal{D}^{\prime }(\omega _{2})$
for a.e $X_{1}\in \omega _{1}$
\end{theorem}

\begin{proof}
Similarly, the proof is a simple application of theorem 9.
\end{proof}

\section{Some Open questions}

\begin{problem}
Suppose that $\infty >p>2$. Given $f\in L^{p}$ and consider (\ref{2}), since 
$f\in L^{2}$then $u_{\epsilon }\rightarrow u_{0}$ in $V_{2}$. Assume that $%
\Omega $ and $A$ are sufficiently regular .Can one prove that $u_{\epsilon
}\rightarrow u_{0}$ in $V_{p}$?
\end{problem}

\begin{problem}
What happens when $f\in L^{1}?$ As mentioned in the introduction there
exists a unique entropy solution to (\ref{2}) which belongs to $%
\dbigcap\limits_{1\leq r<\frac{N}{N-1}}W_{0}^{1,r}(\Omega )$. Can one prove
that $u_{\epsilon }\rightarrow u_{0}$ in $V_{r}$ for some $1\leq r<\frac{N}{%
N-1}?$ Can one prove at least weak convergence in $L^{r}$ for some $1<r<%
\frac{N}{N-1}$ as given in Theorem 4?
\end{problem}

\bigskip

\end{document}